\documentclass[11pt]{amsart}

\usepackage[margin=1.05in]{geometry}
\usepackage{amsmath,amssymb,mathtools,amsthm}
\usepackage{aliascnt}
\usepackage{enumitem}
\usepackage{microtype}
\usepackage{xcolor}
\usepackage[colorlinks=true,linkcolor=blue!55!black,citecolor=blue!55!black,urlcolor=blue!55!black]{hyperref}
\usepackage[nameinlink,capitalize,noabbrev]{cleveref}
\numberwithin{equation}{section}

\newtheorem{theorem}{Theorem}[section]
\newaliascnt{proposition}{theorem}
\newtheorem{proposition}[proposition]{Proposition}
\aliascntresetthe{proposition}
\newaliascnt{lemma}{theorem}
\newtheorem{lemma}[lemma]{Lemma}
\aliascntresetthe{lemma}
\newaliascnt{corollary}{theorem}
\newtheorem{corollary}[corollary]{Corollary}
\aliascntresetthe{corollary}
\theoremstyle{definition}
\newaliascnt{definition}{theorem}

\aliascntresetthe{definition}
\theoremstyle{remark}
\newaliascnt{remark}{theorem}

\aliascntresetthe{remark}

\newcommand{\Cov}{\operatorname{Cov}}
\newcommand{\R}{\mathbb{R}}
\newcommand{\Pp}{\mathbb{P}}
\newcommand{\E}{\mathbb{E}}
\newcommand{\cL}{\mathcal{L}}
\newcommand{\cH}{\mathcal{H}}
\newcommand{\HUT}{\mathcal{H}^{\mathrm{UT}}}
\newcommand{\HUTz}{\mathcal{H}^{\mathrm{UT}}_0}
\newcommand{\KPZ}{\mathfrak{h}}
\newcommand{\Law}{\operatorname{Law}}

\newcommand{\1}{\mathbf{1}}
\newcommand{\eps}{\varepsilon}

\title[THE KPZ UPPER-TAIL FIELD]{The KPZ Upper-Tail Field as a KPZ Fixed Point with Brownian Initial Data}
\author{Ruixuan Zhang}
\address{Department of Mathematics, University of Utah, Salt Lake City, UT 84112, USA}
\date{September 21, 2026.}

\begin{document}

\begin{abstract}
The KPZ upper-tail field introduced by Liu and Zhang arises as the scaling limit of the narrow-wedge KPZ fixed point in a microscopic neighborhood of a conditioned large height. In this paper, we give a direct probabilistic identification of the upper-tail field.  We identify the grounded positive-time sector with the KPZ fixed point started from the two-sided Brownian \(V\)-profile \(B_{\mathrm{ts}}(2x)-2|x|\), while the negative-time sector is represented through an explicit exponential tilt of the same evolution.  The argument uses the directed-landscape variational structure and the local Brownian comparison under upper-tail conditioning.  As consequences of these representations, we recover probabilistically the positive- and negative-time large-scale limits of the KPZ upper-tail field.
\end{abstract}

\maketitle

\section{Introduction}\label{sec:introduction}

The KPZ fixed point \(\KPZ(x,t;h_0)\) of Matetski, Quastel and Remenik \cite{MQR21} and the directed landscape \(\cL(y,s;x,t)\) of Dauvergne, Ortmann and Vir\'ag \cite{DOV22} are two universal objects in the Kardar--Parisi--Zhang universality class.  The KPZ fixed point describes the universal large-scale evolution of one-dimensional random growth models under the characteristic $1{:}2{:}3$ scaling of time, space, and height fluctuations.  The directed landscape provides the corresponding space-time geometry and arises as the universal scaling limit of directed last-passage percolation models.  The directed landscape couples KPZ fixed points started from different initial conditions through the variational formula \(\KPZ(x,t;h_0)=\sup_{y\in\R}\left\{h_0(y)+\cL(y,0;x,t)\right\}.\) Thus the directed landscape couples KPZ fixed points started from different initial conditions on a common probability space and separates the randomness of the initial profile from the subsequent universal evolution. The relation between these objects has recently been sharpened by Dauvergne and Zhang \cite{DZ25}, who characterized the directed landscape from KPZ fixed-point marginals and natural coupling properties.

Upper-tail conditioning provides a natural way to study the geometry of atypically large fluctuations in the KPZ universality class.  Rather than considering only the probability of a rare high value, one may ask how the surrounding space-time field is modified by the conditioning and whether a universal conditional geometry emerges.  Several recent works have addressed this question from different perspectives.  Liu and Wang \cite{LW24} studied the KPZ fixed point before a large value at a prescribed location and obtained a conditional scaling limit with Brownian-bridge behavior.  Nissim and Zhang \cite{NZ22} studied the field after the conditioned point at the one-point level and showed that its distribution is asymptotically described by the one-point law of a refreshed KPZ fixed point, together with an explicit error expansion. Ganguly, Hegde and Zhang \cite{GHZ25} established Brownian-bridge limits for upper-tail path measures and developed conditional comparison estimates that describe the local geometry near a high point.  Related upper-tail questions for the KPZ equation have been studied in \cite{GH22,GLLT23,LT25}, while upper-tail large deviations for the directed landscape and related passage-time objects were investigated in \cite{DDV24,DT24}.  Conditional limits in the periodic KPZ setting were considered in \cite{BL24}.

The two regimes were subsequently unified by Liu and Zhang in \cite{LZ25}.  They considered the narrow-wedge KPZ fixed point conditioned to take a large value at a fixed space-time point and studied the field on the microscopic scale around the conditioned event.  After centering at a large height at $(0,1)$ and rescaling space and time by $L^{-1}$ and $L^{-3/2}$, respectively, they obtained a limiting random field $\HUT(x,t)$ on the full space-time plane, called the KPZ upper-tail field. This field simultaneously records the local fluctuations before and after the conditioned high point and provides a space-time description of the upper-tail environment. Its grounded version
\begin{equation}\label{eq:groundedfield}
	\HUTz(x,t)=\HUT(x,t)-\HUT(0,0)
\end{equation}
has at time zero the exact spatial law
\begin{equation}\label{eq:intro-timezero}
	\left\{\HUTz(x,0):x\in\R\right\}
	\overset{\mathrm{f.d.d.}}{=}
	\left\{B_{\mathrm{ts}}(2x)-2|x|:x\in\R\right\},
\end{equation}
where $\overset{\mathrm{f.d.d.}}{=}$ denotes equality of finite-dimensional distributions and $B_{\mathrm{ts}}$ is a two-sided standard Brownian motion, see \cite[Proposition~1.5(d)]{LZ25}. They further showed that a subsequent large-scale rescaling leads to two different asymptotic regimes.  The negative-time sector converges to a Brownian-type field, while the positive-time sector converges to the narrow-wedge KPZ fixed point \cite[Proposition~1.8]{LZ25}.

The purpose of this paper is to identify the dynamics underlying the upper-tail field.  Our first main result shows that the entire positive-time sector is exactly the KPZ fixed point started from the Brownian $V$-profile in \eqref{eq:intro-timezero}.  Thus the upper-tail conditioning is encoded in the random profile at the distinguished time, while the evolution after that time is given by the usual KPZ fixed-point dynamics.  We then extend this description to finite-dimensional distributions involving negative times.  Using the skew-shift symmetry of Liu and Zhang, we represent the negative-time sector through an exponential change of measure applied to the same positive-time KPZ fixed point.  As a consequence, the positive- and negative-time large-scale limits of \cite[Proposition~1.8]{LZ25} admit direct probabilistic derivations from these representations.

The proofs are probabilistic throughout.  For the positive-time representation, we use the metric composition law and independent increments of the directed landscape together with the conditional Brownian comparison estimates of Ganguly, Hegde and Zhang \cite{GHZ25}.  For the negative-time asymptotics, we combine the tilted representation with one-point Tracy--Widom asymptotics and the local Brownian structure of the directed landscape under upper-tail conditioning.  No exact multipoint formula or asymptotic analysis of such formulas is required.

\subsection{Main results}

Let $B_{\mathrm{ts}}$ be independent of $\cL$ and set
\begin{equation}\label{eq:vprofile}
  h_\vee(x)=B_{\mathrm{ts}}(2x)-2|x|.
\end{equation}
The first theorem identifies the positive-time upper-tail field at its original upper-tail scale.

\begin{theorem}[Positive-time representation]\label{thm:main}
One has
\begin{equation}\label{eq:main-fdd}
  \left\{\HUTz(x,t):(x,t)\in\R\times[0,\infty)\right\}
  \overset{\mathrm{f.d.d.}}{=}
  \left\{\KPZ(x,t;h_\vee):(x,t)\in\R\times[0,\infty)\right\}.
\end{equation}
Equivalently, for every $m\ge1$ and $(x_i,t_i)\in\R\times(0,\infty)$,
\begin{equation}\label{eq:main-explicit}
  \left\{\HUTz(x_i,t_i)\right\}_{i=1}^m
  \overset d=
  \left\{
    \sup_{y\in\R}\left\{B_{\mathrm{ts}}(2y)-2|y|+\cL(y,0;x_i,t_i)\right\}
  \right\}_{i=1}^m.
\end{equation}
The identity holds jointly with any finite collection of time-zero coordinates.
\end{theorem}

The theorem is structural rather than asymptotic.  The upper-tail conditioning is entirely encoded in the random initial profile $h_\vee$; after time zero the evolution is the usual KPZ fixed-point dynamics.  In particular, the positive-time large-scale limit of \cite[Proposition~1.8(b)]{LZ25} follows directly from the $1{:}2{:}3$ scaling of the directed landscape.

\begin{proposition}[Positive-time large-scale limit]\label{prop:positive-longtime}
As $\lambda\to\infty$,
\begin{equation}\label{eq:positive-longtime}
  \left\{\lambda^{-1/3}\HUTz(\lambda^{2/3}x,\lambda t):(x,t)\in\R\times(0,\infty)\right\}
  \Longrightarrow
  \left\{\cL(0,0;x,t):(x,t)\in\R\times(0,\infty)\right\}
\end{equation}
in the sense of convergence of finite-dimensional distributions.
\end{proposition}

We next turn to finite-dimensional distributions involving negative times. Unlike the positive-time sector, these distributions cannot be obtained from an independent directed landscape evolving backward from the distinguished time. Instead, the skew-shift symmetry of the upper-tail field allows the configuration to be re-centered at its earliest observation point. After this re-centering, all observation times are nonnegative and the resulting law is an exponential tilt of the positive-time KPZ fixed point.

Recall the total order $\prec$ from \cite[Definition~1.12]{LZ25}. For $(x,t),(x',t')\in\R^2$, we write $(x,t)\prec(x',t')$ if $t<t'$, or if $t=t'$ and $x<x'$. Let $z_i=(x_i,t_i)$, $1\le i\le m$, be distinct points, none equal to $(0,0)$, and let $z_k=(x_k,t_k)$ be the $\prec$-minimal point. Assume $t_k<0$ and set $s_k=-t_k>0$. Define
\begin{equation}\label{eq:Vk}
	V_k=\KPZ(-x_k,s_k;h_\vee).
\end{equation}
We will prove in \cref{sec:negative} that
$\E_{\Pp}[e^{2V_k}]=e^{2s_k/3}$. Hence
\begin{equation}\label{eq:Qk}
	\frac{d\mathbb Q_k}{d\Pp}
	=
	\exp\left\{
	2V_k-\frac23s_k
	\right\}
\end{equation}
defines a probability measure.

\begin{theorem}[All-time representation]\label{thm:alltime}
	With the notation above,
	\begin{equation}\label{eq:alltime-representation}
		\Law\left(
		\left\{\HUTz(z_i)\right\}_{i=1}^m
		\right)
		=
		\Law_{\mathbb Q_k}\left(
		\left\{
		\KPZ(x_i-x_k, t_i-t_k;h_\vee)-V_k
		\right\}_{i=1}^m
		\right).
	\end{equation}
\end{theorem}

Thus the negative-time sector is not obtained by running the KPZ fixed point backward.  The earliest observation is shifted to time zero and the original grounding point is moved to positive time, producing the exponential tilt \eqref{eq:Qk}.  The measure $\mathbb Q_k$ acts only on the Brownian initial profile and the directed landscape on the right-hand side of \eqref{eq:alltime-representation}.

As a second verification of the representation, we recover the negative-time large-scale limit of \cite[Proposition~1.8(a)]{LZ25} without using the asymptotics of the Liu--Zhang multipoint formula.

\begin{proposition}[Negative-time large-scale limit]\label{prop:negative-longtime}
As $\lambda\to\infty$,
\begin{equation}\label{eq:negative-longtime}
\begin{aligned}
&\left\{
  \frac{\HUTz(\lambda^{1/2}x/\sqrt2,\lambda t)-\lambda t}{\sqrt{2\lambda}}
  :(x,t)\in\R\times(-\infty,0)
\right\}
\\
&\qquad\Longrightarrow
\left\{
  \min\left\{B_1(-t)+x,B_2(-t)-x\right\}
  :(x,t)\in\R\times(-\infty,0)
\right\},
\end{aligned}
\end{equation}
in the sense of convergence of finite-dimensional distributions, where $B_1$ and $B_2$ are independent standard Brownian motions.
\end{proposition}

\subsection{Proof strategy and organization}

The proof of \cref{thm:main} starts from the finite-$L$ variational
decomposition in \cref{lem:finiteL}.  Under the upper-tail conditioning,
\cite[Theorem~1.1 and Proposition~1.5(d)]{LZ25} identify the
finite-dimensional limit of the rescaled profile at the conditioned
time as $h_\vee$, while \cref{lem:finiteL} shows that the rescaled
future landscape is independent of this profile and has exactly the
directed-landscape law.  Finite-dimensional convergence of the initial
profile is not sufficient to pass directly through the variational
supremum.  The local oscillation estimate in \cref{prop:II} upgrades
the convergence to the compact uniform topology, which gives
\cref{cor:profile-compact}.  The localization estimate in
\cref{prop:III} then shows that the maximizing point remains in a fixed
compact set with probability tending to one.  These two estimates rule
out, respectively, narrow fluctuations between deterministic
observation points and escape of the maximizer to spatial infinity.
Combining \cref{lem:finiteL,cor:profile-compact,prop:III} with the
variational formula for the KPZ fixed point proves \cref{thm:main}.
The same representation and the $1{:}2{:}3$ scaling of the directed
landscape yield \cref{prop:positive-longtime}.

The proof of \cref{thm:alltime} uses the shift structure of the
upper-tail field established in \cite{LZ25}.  The simultaneous shift
identity \cite[Proposition~2.8]{LZ25} moves the earliest observation
point to the origin, so that all shifted observation times are
nonnegative.  The derivative formula
\cite[Proposition~2.12]{LZ25} for the grounded field is combined with
the elementary exponential differentiation argument preceding
\cref{lem:skew-transfer}.  This converts the derivative in the shifted
threshold into the exponential weight appearing in \eqref{eq:Qk}.
The resulting tail identity is stated in \cref{lem:skew-transfer}.
Applying the positive-time representation of \cref{thm:main} to the
shifted configuration and normalizing the exponential weight then
gives the all-time representation in \cref{thm:alltime}.

The proof of \cref{prop:negative-longtime} begins from this tilted
representation.  Two elementary consequences of the exponential tilt
are isolated first.  Lemma~\ref{lem:neg-Brownian-tilt} describes the
Brownian initial profile under the spatial tilt, while
\cref{lem:neg-tilted-GUE} shows that a one-point directed-landscape
height under the corresponding exponential tilt has Gaussian
fluctuations on the scale $\lambda^{1/2}$.  The main step is the
finite-chain replacement in \cref{lem:neg-chain-replacement}, which
uses the local Brownian comparison of
\cite[Proposition~3.5]{GHZ25}

\section{Preliminaries}\label{sec:objects}

\subsection{The directed landscape and the KPZ fixed point}

The directed landscape $\cL$ is a random continuous function on \(\R_{\uparrow}^4=\left\{(x,s;y,t)\in\R^4 \,:\, s<t\right\},\)
which arises as the universal scaling limit of directed last-passage geometries in the KPZ universality class.  It is characterized by independent increments on disjoint time intervals, the metric composition law, and Airy-sheet marginals on each fixed time interval.  We refer to \cite{DOV22} for its construction and basic properties.

For an admissible initial profile $h_0:\R\to\R\cup\left\{-\infty\right\}$, the KPZ fixed point started from $h_0$ can be realized through the directed landscape by
\begin{equation}\label{eq:kpz-variational-prelim}
	\KPZ(x,t;h_0)
	=
	\sup_{y\in\R}
	\left\{
		h_0(y)+\cL(y,0;x,t)
		\right\},
	\qquad t>0.
\end{equation}
For narrow-wedge initial data $h_{\mathrm{nw}}$, defined by $h_{\mathrm{nw}}(0)=0$ and $h_{\mathrm{nw}}(x)=-\infty$ for $x\ne0$, this reduces to
\begin{equation}\label{eq:narrow-wedge}
	\cH(x,t)
	=
	\KPZ(x,t;h_{\mathrm{nw}})
	=
	\cL(0,0;x,t).
\end{equation}

We will use several standard properties of the directed landscape throughout the paper and we refer to \cite{DOV22} for further details.

\begin{itemize}
	
	\item \emph{Metric composition.}
	Almost surely, for every $s<r<t$ and $x,y\in\R$,
	\begin{equation}\label{eq:metric-composition}
		\cL(x,s;y,t)
		=
		\sup_{z\in\R}
		\left\{
		\cL(x,s;z,r)+\cL(z,r;y,t)
		\right\}.
	\end{equation}
	
	\item \emph{Independent increments.}
	If $(s_i,t_i)$, $1\le i\le n$, are pairwise disjoint time intervals, then
	the random continuous functions
	\[
	\left\{
	(x,y)\longmapsto \cL(x,s_i;y,t_i)
	\right\}_{i=1}^n
	\]
	are independent.
	
	\item \emph{Space-time translation invariance.}
	For every $\alpha,\eta\in\R$,
	\begin{equation}\label{eq:landscape-translation}
		\left\{
		\cL(x+\alpha,s+\eta;y+\alpha,t+\eta)
		\right\}_{x,y\in\R,\ s<t}
		\overset d=
		\left\{
		\cL(x,s;y,t)
		\right\}_{x,y\in\R,\ s<t}.
	\end{equation}
	
	\item \emph{$1{:}2{:}3$ scaling.}
	For every $a>0$,
	\begin{equation}\label{eq:landscape-scaling}
		\left\{
		a^{-1}\cL(a^2x,a^3s;a^2y,a^3t)
		\right\}_{x,y\in\R,\ s<t}
		\overset d=
		\left\{
		\cL(x,s;y,t)
		\right\}_{x,y\in\R,\ s<t}.
	\end{equation}
	
	\item \emph{One-point distribution.}
	For every $x,y\in\R$ and $s<t$,
	\begin{equation}\label{eq:landscape-one-point}
		\cL(x,s;y,t)
		\overset d=
		(t-s)^{1/3}\chi-\frac{(x-y)^2}{t-s},
	\end{equation}
	where \(\chi\) is a Tracy-Widom GUE random variable.
	
\item \emph{Global parabolic bound.}
For every fixed $s<t$ and every $\varepsilon>0$, there exists an almost surely finite random constant $C=C(s,t,\varepsilon)$ such that
\begin{equation}\label{eq:global-parabolic}
	\left|
	\cL(x,s;y,t)
	+\frac{(x-y)^2}{t-s}
	\right|
	\le
	C\left(1+|x|+|y|\right)^\varepsilon
\end{equation}
for all $x,y\in\R$.  This follows from \cite[Corollary~10.7]{DOV22}.
\end{itemize}

We will also use the standard upper-tail asymptotics of the GUE Tracy--Widom distribution, see \cite[(3.10)]{LZ25}.  As $L\to\infty$,
\begin{equation}\label{eq:upper-tail-GUE}
	1-F_{\rm GUE}(L)
	=
	\frac{e^{-\frac43L^{3/2}}}{16\pi L^{3/2}}
	\left(1+O(L^{-3/2})\right),
	\qquad
	F_{\rm GUE}'(L)
	=
	\frac{e^{-\frac43L^{3/2}}}{8\pi L}
	\left(1+O(L^{-3/2})\right).
\end{equation}

\subsection{The upper-tail field}

For $L>0$, let
\begin{equation}\label{eq:HL-definition}
  H_L(x,t)=\sqrt L\,[\cH(L^{-1}x,1+L^{-3/2}t)-L],
  \qquad
  A_L=\left\{\cH(0,1)\ge L\right\},
\end{equation}
whenever $1+L^{-3/2}t>0$.  By \cite[Theorem~1.1]{LZ25}, under $\Pp(\,\cdot\mid A_L)$,
\begin{equation}\label{eq:HUT-definition}
  \left\{H_L(x,t):(x,t)\in\R^2\right\}
  \Longrightarrow
  \left\{\HUT(x,t):(x,t)\in\R^2\right\}
\end{equation}
in finite-dimensional distributions.  We write
\begin{equation}\label{eq:HUT-grounded}
  \HUTz(x,t)=\HUT(x,t)-\HUT(0,0).
\end{equation}
By \cite[Proposition~1.5(a),(c),(d)]{LZ25}, $\HUT(0,0)$ is exponential with parameter $2$, is independent of $\HUTz$, and \eqref{eq:intro-timezero} holds.  Throughout the paper, $\Pr$ denotes probability under the intrinsic law of the upper-tail field, while $\Pp$ is the probability measure carrying $B_{\mathrm{ts}}$ and the directed landscape.

\section{Positive-time representation}\label{sec:positiveproofs}

Set
\begin{equation}\label{eq:gL}
  g_L(y)=\sqrt L\,[\cH(L^{-1}y,1)-\cH(0,1)]
\end{equation}
and, for $0\le s<t$,
\begin{equation}\label{eq:futurelandscape}
  \cL_L^+(y,s;x,t)
  =\sqrt L\,\cL(L^{-1}y,1+L^{-3/2}s;L^{-1}x,1+L^{-3/2}t).
\end{equation}
By \eqref{eq:HUT-definition}, \eqref{eq:HUT-grounded}, and \eqref{eq:intro-timezero}, for every fixed $y_1,\ldots,y_m$,
\begin{equation}\label{eq:gL-fdd}
  \left\{g_L(y_i)\right\}_{i=1}^m
  \Longrightarrow
  \left\{h_\vee(y_i)\right\}_{i=1}^m
\end{equation}
under $\Pp(\,\cdot\mid A_L)$.

The future evolution admits an exact decomposition.

\begin{lemma}[Finite-$L$ decomposition]\label{lem:finiteL}
For $t>0$,
\begin{equation}\label{eq:exact-var}
  \sqrt L\,[\cH(L^{-1}x,1+L^{-3/2}t)-\cH(0,1)]
  =\sup_{y\in\R}\left\{g_L(y)+\cL_L^+(y,0;x,t)\right\}.
\end{equation}
Moreover, \(\cL_L^+\overset d=\cL,\) and $\cL_L^+$ is independent of $(g_L,A_L)$.
\end{lemma}

\begin{proof}
	By the metric composition law \eqref{eq:metric-composition} at time $1$,
	\[
	\cH(L^{-1}x,1+L^{-3/2}t)
	=
	\sup_{z\in\R}
	\left\{
	\cH(z,1)
	+
	\cL(z,1;L^{-1}x,1+L^{-3/2}t)
	\right\}.
	\]
	Subtracting $\cH(0,1)$, setting $z=L^{-1}y$, and multiplying by
	$\sqrt L$ gives \eqref{eq:exact-var}.
	
	By the space-time translation invariance \eqref{eq:landscape-translation}
	and the $1{:}2{:}3$ scaling \eqref{eq:landscape-scaling}, yields \(\cL_L^+\overset d=\cL.\) Finally, $g_L$ and $A_L$ are measurable with respect to the directed landscape up to time $1$, while $\cL_L^+$ depends only on increments after
	time $1$.  The independent-increment property therefore implies that $\cL_L^+$ is independent of $(g_L,A_L)$.
\end{proof}

The convergence \eqref{eq:gL-fdd} cannot be inserted directly into \eqref{eq:exact-var}.  A variational supremum depends on the full spatial profile, whereas finite-dimensional convergence does not exclude narrow peaks between the observation points, and local uniform convergence alone does not prevent the maximizer from escaping to infinity. We therefore prove, respectively, compact tightness and localization.

For $R<\infty$, let
\begin{equation}
	\omega_R(f,\delta)=\sup\left\{|f(y)-f(z)|:y,z\in[-R,R],\ |y-z|\le\delta\right\}.
\end{equation}

\begin{proposition}[Local oscillation]\label{prop:II}
For every $R<\infty$ and $\eps>0$,
\begin{equation}\label{eq:II}
  \lim_{\delta\downarrow0}\limsup_{L\to\infty}
  \Pp\left(\omega_R(g_L,\delta)>\eps\mid A_L\right)=0.
\end{equation}
\end{proposition}

Since $g_L(0)=0$, \cref{prop:II} gives tightness of the conditional laws of $g_L$ in $C([-R,R])$ equipped with the uniform topology.  By Prokhorov's theorem, every subsequence therefore contains a further subsequence converging weakly in $C([-R,R])$.  The finite-dimensional convergence \eqref{eq:gL-fdd} implies that every such subsequential limit has the same finite-dimensional distributions as $h_\vee$.  Since $h_\vee$ has continuous sample paths, these finite-dimensional distributions determine its law on $C([-R,R])$.  This proves the following compact convergence statement.

\begin{corollary}[Compact convergence]\label{cor:profile-compact}
	For every $R<\infty$, under the conditional law
	$\Pp(\,\cdot\mid A_L)$,
	\begin{equation}\label{eq:profile-compact}
		g_L
		\Longrightarrow
		h_\vee
	\end{equation}
	in distribution in $C([-R,R])$ as \(L \to \infty, \) equipped with the uniform topology.
\end{corollary}

To apply \cref{cor:profile-compact} to the variational formula \eqref{eq:exact-var}, we must also show that the
maximizing point does not escape to spatial infinity.  Since $y=0$ is an
admissible point in the variational problem, it is enough to show that,
with probability tending to one, no point outside a sufficiently large
compact interval can produce a value larger than the value at $y=0$.
This is the content of the following localization estimate.

\begin{proposition}[Localization]\label{prop:III}
For every finite collection $(x_i,t_i)\in\R\times(0,\infty)$,
\begin{equation}\label{eq:III}
  \lim_{R\to\infty}\limsup_{L\to\infty}
  \Pp\left(
    \left.
    \exists i:\ \sup_{|y|>R}\left\{g_L(y)+\cL_L^+(y,0;x_i,t_i)\right\}
    \ge \cL_L^+(0,0;x_i,t_i)
    \ \right|\ A_L
  \right)=0.
\end{equation}
\end{proposition}

\subsection{Local oscillation}

We use two inputs concerning the spatial profile under the upper-tail
conditioning.  The first controls the height at the conditioned point.
By \cite[Remark~1.6, equation~(1.13)]{LZ25},
\begin{equation}\label{eq:overshoot}
	\sqrt L\left[\cH(0,1)-L\right]
	=
	O_{\Pp(\,\cdot\mid A_L)}(1),
	\qquad
	A_L=\left\{\cH(0,1)\ge L\right\}.
\end{equation}
In fact, the left-hand side converges to an $\operatorname{Exp}(2)$
random variable.

The second input is the zero-temperature specialization of
\cite[Proposition~3.5]{GHZ25}.  Set $H=10^{-6}L^{1/2}$ and, for
$0\le u\le H$, write
\[
P_L^\pm(u)
=
\cH(\pm u,1)-\cH(0,1).
\]
Let $\mu_L$ be the law of $(P_L^+,P_L^-)$ under
$\Pp(\,\cdot\mid A_L)$.  Let $B_+$ and $B_-$ be independent rate-$2$
Brownian bridges on $[0,H]$ from $0$ to $0$, independent of the
directed landscape, and set
\[
\widetilde P_L^\pm(u)
=
B_\pm(u)
+
\frac{u}{H}
\left[\cH(\pm H,1)-\cH(0,1)\right].
\]
Thus $\widetilde P_L^\pm$ is the Brownian bridge obtained by adding the
affine interpolation between the two endpoint values of $P_L^\pm$.

Proposition~3.5 of \cite{GHZ25}, applied with $L^+=\infty$ and then
restricted to these two profiles, gives an auxiliary probability
measure $\nu_L$ on $C([0,H],\R)^2$.  Outside an exceptional event of
probability at most $Ce^{-cL^{3/2}}$, its Radon--Nikodym derivative
with respect to $\mu_L$ satisfies
\[
\frac{d\nu_L}{d\mu_L}
=
1+O(e^{-cL}),
\]
and $\nu_L$ can be coupled with the law of
$(\widetilde P_L^+,\widetilde P_L^-)$ so that
\[
\max_{\sigma\in\left\{+,-\right\}}
\sup_{0\le u\le H}
\left|
P_L^\sigma(u)-\widetilde P_L^\sigma(u)
\right|
\le
Ce^{-cL}
\]
outside an event of probability at most $Ce^{-cL^{3/2}}$.  Hence,
up to errors that are exponentially small in $L$, the conditioned
profiles are independent rate-$2$ Brownian bridges around their random
affine endpoint interpolations.

We control these affine terms with the upper-tail tent estimate
\cite[Lemma~2.20]{GHZ25}.  For
$I\subset[-L^{1/2}/2,L^{1/2}/2]$, put
$\sigma_I=\sup_{u\in I}|u|^{1/2}$.  Then, for
$0<M<L^{3/4}$,
\begin{equation}\label{eq:GHZtent}
	\Pp\left(
	\left.
	\sup_{u\in I}
	\left|
	\cH(u,1)
	-\left(L-2L^{1/2}|u|\right)
	\right|
	>M\sigma_I
	\ \right|\ A_L
	\right)
	\le
	C\exp\left\{
	-c\left(
	M^2\wedge M\sigma_I L^{1/2}
	\right)
	\right\}.
\end{equation}
Applying \eqref{eq:GHZtent} at the endpoints $\pm H$, for example with
$M=L^{1/8}$, gives
\[
\frac{
	\cH(\pm H,1)-\left(L-2L^{1/2}H\right)
}{
	H\sqrt L
}
\longrightarrow0
\]
in probability under $\Pp(\,\cdot\mid A_L)$.  Together with
\eqref{eq:overshoot}, this yields
\begin{equation}\label{eq:endpoint-slope}
	\frac{
		\cH(\pm H,1)-\cH(0,1)
	}{
		H\sqrt L
	}
	\longrightarrow-2
\end{equation}
in probability.  Thus Proposition~3.5 of \cite{GHZ25} controls the
Brownian fluctuations around the affine interpolation, while
\eqref{eq:GHZtent} identifies its limiting slope.

\begin{proof}[Proof of \cref{prop:II}]
	Fix $R<\infty$.  Since $R/L<H$ for all sufficiently large $L$, the
	comparison above applies to the restrictions of the profiles to
	$[0,R/L]$.  For $0\le y\le R$,
	\[
	g_L(y)
	=
	\sqrt L\,P_L^+(y/L).
	\]
	The Radon--Nikodym comparison and the coupling with
	$\widetilde P_L^+$ therefore imply
	\begin{equation}\label{eq:bridge-rescaled}
		\sup_{0\le y\le R}
		\left|
		g_L(y)
		-\sqrt L\,B_+(y/L)
		-\frac{\cH(H,1)-\cH(0,1)}{H\sqrt L}\,y
		\right|
		\longrightarrow0
	\end{equation}
	in probability under $\Pp(\,\cdot\mid A_L)$.
	
	The rescaled Brownian bridges are tight in $C([0,R])$.  Indeed, if
	$0\le y,z\le R$, then
	\[
	\operatorname{Cov}\left(
	\sqrt L\,B_+(y/L),
	\sqrt L\,B_+(z/L)
	\right)
	=
	2\left(
	\min\left\{y,z\right\}
	-\frac{yz}{LH}
	\right),
	\]
	and hence
	\begin{equation}\label{eq:bridge-limit}
		\left\{
		\sqrt L\,B_+(y/L)
		:0\le y\le R
		\right\}
		\Longrightarrow
		\left\{
		B(2y):0\le y\le R
		\right\}
	\end{equation}
	in $C([0,R])$, where $B$ is a standard Brownian motion.
	
	By \eqref{eq:endpoint-slope}, the affine coefficient in
	\eqref{eq:bridge-rescaled} is tight and converges to $-2$.  Therefore,
	for every $\delta>0$,
	\[
	\omega_R(g_L,\delta)
	\le
	\omega_R\left(
	\sqrt L\,B_+(\,\cdot/L),\delta
	\right)
	+
	\delta
	\left|
	\frac{\cH(H,1)-\cH(0,1)}
	{H\sqrt L}
	\right|
	+
	o_{\Pp(\,\cdot\mid A_L)}(1).
	\]
	Using \eqref{eq:bridge-limit} and the almost-sure continuity of Brownian
	motion gives, for every $\eps>0$,
	\[
	\lim_{\delta\downarrow0}
	\limsup_{L\to\infty}
	\Pp\left(
	\left.
	\omega_R(g_L,\delta)>\eps
	\ \right|\ A_L
	\right)
	=0
	\]
	on $[0,R]$.  The same argument applied to $P_L^-$ gives the corresponding
	estimate on $[-R,0]$.  Since $g_L(0)=0$, the two estimates combine to
	give \eqref{eq:II} on $[-R,R]$.
\end{proof}

\subsection{Localization}

For the outer region we use the global parabolic estimate in \cite[Lemma~3.2]{GHZ25}, there exist
deterministic constants $C,c>0$ and an almost surely finite nonnegative
random variable $C_{\mathrm{par}}$ such that
\begin{equation}\label{eq:initial-global}
	\left|
	\cH(z,1)+z^2
	\right|
	\le
	C_{\mathrm{par}}+\log(2+|z|),
	\qquad z\in\R,
\end{equation}
and, for every $M>0$,
\begin{equation}\label{eq:initial-global-tail}
	\Pp\left(
	C_{\mathrm{par}}>M
	\right)
	\le
	Ce^{-cM^{3/2}}.
\end{equation} 
By \cite[Corollary~10.7]{DOV22}, for fixed $(x,t)$ there is a random variable $K_{x,t}$ with the same type of stretched-exponential tail such that
\begin{equation}\label{eq:future-global}
  \cL(y,0;x,t)
  \le-\frac{(y-x)^2}{t}+K_{x,t}+C_t\log(2+|y|).
\end{equation}

\begin{proof}[Proof of \cref{prop:III}]
On $A_L$,
\begin{equation}\label{eq:g-upper-by-L}
  g_L(y)\le\sqrt L\,[\cH(L^{-1}y,1)-L].
\end{equation}
Fix $c_*=1/8$.  Let $r\ge1$ satisfy
$2r\le c_*L^{3/2}$ and apply \eqref{eq:GHZtent} with \(I=\left[\frac rL,\frac{2r}{L}\right], M=\frac{\sqrt r}{2\sqrt2}.\)
Then \(\sigma_I=\sqrt{\frac{2r}{L}}, M\sigma_I=\frac{r}{2\sqrt L},\) and hence \(M^2=\frac r8, M\sigma_I L^{1/2}=\frac r2.\)
Therefore \eqref{eq:GHZtent} gives
\begin{equation}\label{eq:tent-annulus-original}
	\Pp\left(
	\left.
	\sup_{r/L\le u\le2r/L}
	\left|
	\cH(u,1)-\left(L-2L^{1/2}u\right)
	\right|
	>
	\frac{r}{2\sqrt L}
	\ \right|\ A_L
	\right)
	\le Ce^{-cr}.
\end{equation}
On the complement of this event, for $r\le y\le2r$, \(\sqrt L\left[\cH(L^{-1}y,1)-L\right]\le -2y+\frac r2 \le-\frac32r.\)
Since on $A_L$, \(g_L(y)=\sqrt L\left[\cH(L^{-1}y,1)-\cH(0,1)\right]\le\sqrt L\left[\cH(L^{-1}y,1)-L\right].\)
Applying the same argument to the reflected interval gives
\begin{equation}\label{eq:annulus}
	\Pp\left(
	\left.
	\sup_{r\le|y|\le2r}g_L(y)>-\frac32r
	\ \right|\ A_L
	\right)
	\le Ce^{-cr}.
\end{equation}

To pass from \eqref{eq:annulus} to a uniform bound, apply it with
$r=2^jR$ over all dyadic annuli satisfying
$2^{j+1}R\le c_*L^{3/2}$.  On the complement of the corresponding
exceptional events,
\[
g_L(y)\le-\frac32\,2^jR
\le-\frac34|y|,
\qquad
2^jR\le|y|\le2^{j+1}R.
\]
Moreover, \(\sum_{j\ge0}Ce^{-c2^jR} \le C'e^{-c'R}.\)
The possible terminal interval between the last dyadic annulus and
$c_*L^{3/2}$ is covered by one further application of
\eqref{eq:annulus} with $r=c_*L^{3/2}/2$.  Consequently,
\begin{equation}\label{eq:tent-linear}
	\Pp\left(
	\left.
	g_L(y)\le-c_1|y|
	\ \text{for }R\le|y|\le c_*L^{3/2}
	\ \right|\ A_L
	\right)
	\ge1-Ce^{-c_2R},
\end{equation}
where one may take $c_1=3/4$.
By \cref{lem:finiteL}, the global parabolic estimate
\eqref{eq:future-global} applies to $\cL_L^+$ with constants
uniform in $L$.  More precisely, for each $(x_i,t_i)$ there is a
nonnegative random variable $K^{(L)}_{x_i,t_i}$, independent of $A_L$,
such that
\[
\Pp\left(K^{(L)}_{x_i,t_i}>M\right)\le Ce^{-cM^{3/2}},
\qquad M>0,
\]
uniformly in $L$, and
\[
\cL_L^+(y,0;x_i,t_i)
\le-\frac{(y-x_i)^2}{t_i}+K^{(L)}_{x_i,t_i}+C_i\log(2+|y|).
\]
Hence, on the event in \eqref{eq:tent-linear},
\begin{equation}
	g_L(y)+\cL_L^+(y,0;x_i,t_i)
	\le
	-c_1|y|
	-\frac{(y-x_i)^2}{t_i}
	+K^{(L)}_{x_i,t_i}
	+C_i\log(2+|y|).
\end{equation}
Since \((y-x_i)^2\ge\frac12y^2-x_i^2,\) there exist $c_3>0$ and $C_i<\infty$, independent of $L$, such that
for all sufficiently large $R$,
\begin{equation}\label{eq:intermediate-decay}
	\sup_{R\le|y|\le c_*L^{3/2}}
	\left\{
	g_L(y)+\cL_L^+(y,0;x_i,t_i)
	\right\}
	\le
	-c_3R^2+K^{(L)}_{x_i,t_i}+C_i.
\end{equation}

For the region $|y|\ge c_*L^{3/2}$, we first control the random
constant in \eqref{eq:initial-global}.  Since \(A_L=\left\{\cH(0,1)\ge L\right\}\)
and $\cH(0,1)$ has the GUE Tracy--Widom distribution,  \eqref{eq:upper-tail-GUE} gives
\begin{equation}\label{eq:AL-asymptotic}
	\Pp(A_L)
	=
	\frac{e^{-\frac43L^{3/2}}}{16\pi L^{3/2}}
	\left(1+O(L^{-3/2})\right).
\end{equation}
On the other hand, \eqref{eq:initial-global-tail} implies that for every
$K>0$, \(\Pp\left(C_{\mathrm{par}}>KL\right)\le C\exp\left\{-cK^{3/2}L^{3/2}\right\}.\)
Therefore
\begin{align}
	\Pp\left(
	C_{\mathrm{par}}>KL
	\,\middle|\,
	A_L
	\right)
	\le
	\frac{\Pp\left(C_{\mathrm{par}}>KL\right)}{\Pp(A_L)}
	\le
	C L^{3/2}
	\exp\left\{
	-\left(cK^{3/2}-\frac43\right)L^{3/2}
	\right\}
	\label{eq:Cpar-cond-bound}
\end{align}
for all sufficiently large $L$.  Choosing $K$ so that
$cK^{3/2}>4/3$ yields
\begin{equation}\label{eq:parabolic-constant-cond}
	\Pp\left(
	C_{\mathrm{par}}>KL
	\,\middle|\,
	A_L
	\right)
	\longrightarrow0.
\end{equation}

We next derive the corresponding bound for $g_L$.  On $A_L$,
\eqref{eq:g-upper-by-L} gives
\[
g_L(y)\le\sqrt L\left[\cH(L^{-1}y,1)-L\right].
\]
Applying \eqref{eq:initial-global} with $z=L^{-1}y$ gives
\[
\cH(L^{-1}y,1)
\le
-\frac{y^2}{L^2}+C_{\mathrm{par}}
+\log\left(2+\frac{|y|}{L}\right).
\]
Hence 
\[g_L(y)
\le
-\frac{y^2}{L^{3/2}}
+\sqrt L\,C_{\mathrm{par}}
-L^{3/2}
+\sqrt L\log\left(2+\frac{|y|}{L}\right).\]
On the event
$\left\{C_{\mathrm{par}}\le KL\right\}$, this becomes
\begin{equation}\label{eq:outer-g-bound}
	g_L(y)
	\le
	-\frac{y^2}{L^{3/2}}
	+(K-1)L^{3/2}
	+\sqrt L\log\left(2+\frac{|y|}{L}\right).
\end{equation}

By \cref{lem:finiteL}, $\cL_L^+$ has the directed-landscape law and is
independent of $A_L$.  Hence the global parabolic estimate
\eqref{eq:future-global} applies to $\cL_L^+$ with the same tail bounds.
For the finite collection $(x_i,t_i)$, with conditional probability
tending to one,
\begin{equation}\label{eq:outer-future-bound}
	\cL_L^+(y,0;x_i,t_i)
	\le
	-\frac{(y-x_i)^2}{t_i}
	+L+C_i\log(2+|y|)
\end{equation}
simultaneously for all $i$ and all $y\in\R$.

For $|y|\ge c_*L^{3/2}$ and $L$ sufficiently large, \(\frac{(y-x_i)^2}{t_i}\ge \frac{y^2}{2t_i}.\)
Moreover,
\begin{equation}
	\frac{L^{3/2}}{y^2}
	\le
	\frac{1}{c_*^2L^{3/2}},
	\qquad
	\frac{\sqrt L\log\left(2+|y|/L\right)}{y^2}
	=o(1),
\end{equation}
uniformly for $|y|\ge c_*L^{3/2}$, and the same holds for
$L/y^2$ and $\log(2+|y|)/y^2$.  Combining
\eqref{eq:outer-g-bound} and \eqref{eq:outer-future-bound} therefore
gives, 
\begin{equation}\label{eq:outer-decay}
	\Pp\left(
	\left.
	\exists\,1\le i\le m,\ \exists\,|y|\ge c_*L^{3/2}
	\text{ such that }
	g_L(y)+\cL_L^+(y,0;x_i,t_i)>-c_4y^2
	\ \right|\ A_L
	\right)
	\longrightarrow0,
\end{equation}
where one may take \(c_4=\frac{1}{4\max_{1\le i\le m} t_i}.\)

For the intermediate region, \eqref{eq:intermediate-decay} and the
uniform tail bound on $K^{(L)}_{x_i,t_i}$ imply, for every $M>0$,
\[
\lim_{R\to\infty}\limsup_{L\to\infty}
\Pp\left(
\left.
\exists i:\ \sup_{R\le|y|\le c_*L^{3/2}}
\left\{g_L(y)+\cL_L^+(y,0;x_i,t_i)\right\}>-M
\ \right|\ A_L
\right)=0.
\]
Together with \eqref{eq:outer-decay}, this gives, for every $M>0$,
\begin{equation}\label{eq:outside-minus-M}
  \lim_{R\to\infty}\limsup_{L\to\infty}
  \Pp\left(
    \left.
    \exists i:\ \sup_{|y|>R}\left\{g_L(y)+\cL_L^+(y,0;x_i,t_i)\right\}>-M
    \ \right|\ A_L
  \right)=0.
\end{equation}
Since $g_L(0)=0$ and 
\(\left\{
\cL_L^+(0,0;x_i,t_i)
\right\}_{i=1}^m
\overset d=
\left\{
\cL(0,0;x_i,t_i)
\right\}_{i=1}^m,\)
if the event in \eqref{eq:III} occurs and all reference values
$\cL_L^+(0,0;x_i,t_i)\ge -M,$ then the corresponding
outer supremum is at least $-M$.  Hence, for every $M>0$,
\begin{align}
	&\Pp\left(
	\left.
	\exists\,1\le i\le m,\ 
	\sup_{|y|>R}
	\left\{
	g_L(y)+\cL_L^+(y,0;x_i,t_i)
	\right\}
	\ge
	\cL_L^+(0,0;x_i,t_i)
	\ \right|\ A_L
	\right)
	\nonumber\\
	&\qquad\le
	\Pp\left(
	\left.
	\exists\,1\le i\le m,\ 
	\sup_{|y|>R}
	\left\{
	g_L(y)+\cL_L^+(y,0;x_i,t_i)
	\right\}
	\ge -M
	\ \right|\ A_L
	\right)
	\nonumber\\
	&\qquad\quad+
	\Pp\left(
	\min_{1\le i\le m}
	\cL(0,0;x_i,t_i)<-M
	\right).
	\label{eq:localization-final-bound}
\end{align}
Therefore,
\begin{align*}
	&\lim_{R\to\infty}\limsup_{L\to\infty}
	\Pp\left(
	\left.
	\exists\,1\le i\le m,\ 
	\sup_{|y|>R}
	\left\{
	g_L(y)+\cL_L^+(y,0;x_i,t_i)
	\right\}
	\ge
	\cL_L^+(0,0;x_i,t_i)
	\ \right|\ A_L
	\right)
	\\
	&\qquad\le
	\lim_{M\to\infty}
	\Pp\left(
	\min_{1\le i\le m}
	\cL(0,0;x_i,t_i)<-M
	\right)
	=0.
\end{align*}
\end{proof}

\subsection{Proof of the representation and the positive-time limit}

\begin{proof}[Proof of \cref{thm:main}]
Fix $(x_i,t_i)\in\R\times(0,\infty)$, $1\le i\le m$.  For every fixed $R$, \cref{cor:profile-compact} gives
$g_L\Longrightarrow h_\vee$ in $C([-R,R])$ under $\Pp(\,\cdot\mid A_L)$.  By \cref{lem:finiteL}, $\cL_L^+$ is independent of $(g_L,A_L)$ and has exactly the directed-landscape law.  Hence, under the conditional law,
\[
\left(g_L,\cL_L^+\right)
\Longrightarrow
\left(h_\vee,\cL\right)
\]
on compact sets, where $h_\vee$ and $\cL$ are independent.  The continuous mapping theorem therefore gives, for every fixed $R$,
\begin{equation}\label{eq:compact-sup-limit}
  \left\{
    \sup_{|y|\le R}\left\{g_L(y)+\cL_L^+(y,0;x_i,t_i)\right\}
  \right\}_{i=1}^m
  \Longrightarrow
  \left\{
    \sup_{|y|\le R}\left\{h_\vee(y)+\cL(y,0;x_i,t_i)\right\}
  \right\}_{i=1}^m.
\end{equation}
By \cref{prop:III}, the compact restriction can be removed after $L\to\infty$ and then $R\to\infty$.  The limiting suprema are finite by \cite[Corollary~10.7]{DOV22}. Hence \eqref{eq:kpz-variational-prelim} gives
\begin{equation}\label{eq:prelimit-to-kpz}
  \left\{
    \sqrt L\,[\cH(L^{-1}x_i,1+L^{-3/2}t_i)-\cH(0,1)]
  \right\}_{i=1}^m
  \Longrightarrow
  \left\{\KPZ(x_i,t_i;h_\vee)\right\}_{i=1}^m.
\end{equation}
By \cite[Theorem~1.1]{LZ25}, the left-hand side of \eqref{eq:prelimit-to-kpz} converges jointly to $\left\{\HUTz(x_i,t_i)\right\}_{i=1}^m$.  Uniqueness of weak limits proves \eqref{eq:main-explicit}.  The same argument, together with \eqref{eq:profile-compact}, is joint with finitely many time-zero coordinates and proves \eqref{eq:main-fdd}.
\end{proof}

\begin{proof}[Proof of \cref{prop:positive-longtime}]
By \cref{thm:main}, \eqref{eq:landscape-scaling}, and Brownian scaling,
\begin{equation}\label{eq:positive-scaled-variational}
  \lambda^{-1/3}\HUTz(\lambda^{2/3}x,\lambda t)
  \overset d=
  \sup_{y\in\R}\left\{
    B_{\mathrm{ts}}(2y)-2\lambda^{1/3}|y|+\cL(y,0;x,t)
  \right\}.
\end{equation}
By \cite[Corollary~10.7]{DOV22} and $B_{\mathrm{ts}}(2y)=o(|y|)$ almost surely,
\[
  \sup_{y\in\R}\left\{B_{\mathrm{ts}}(2y)+\cL(y,0;x,t)-|y|\right\}<\infty
  \qquad\text{a.s.}
\]
Hence for every $\delta>0$, the right-hand side of \eqref{eq:positive-scaled-variational} restricted to $|y|\ge\delta$ tends uniformly to $-\infty$.  Since its value at $y=0$ is $\cL(0,0;x,t)$,
\begin{align*}
0
&\le \sup_{y\in\R}\left\{B_{\mathrm{ts}}(2y)-2\lambda^{1/3}|y|+\cL(y,0;x,t)\right\}-\cL(0,0;x,t)
\\
&\le \sup_{|y|<\delta}\left|B_{\mathrm{ts}}(2y)+\cL(y,0;x,t)-\cL(0,0;x,t)\right|
\end{align*}
for all sufficiently large $\lambda$.  Let $\lambda\to\infty$ and then $\delta\downarrow0$.  Continuity of $B_{\mathrm{ts}}$ and $\cL$ gives almost-sure convergence to $\cL(0,0;x,t)$, jointly for every finite collection of $(x,t)$.  This proves \eqref{eq:positive-longtime}.
\end{proof}

\section{All-time representation}\label{sec:negative}

We now prove \cref{thm:alltime}.  The only inputs from \cite{LZ25} are
the simultaneous shift relation for the joint tail probabilities of the
upper-tail field and the derivative formula for the grounded field.
To simplify notation, we write
\begin{equation}\label{eq:T-definition}
	T(\beta_1,\ldots,\beta_m;z_1,\ldots,z_m)
	=
	\Pr\left(
	\bigcap_{i=1}^m
	\left\{
	\HUT(z_i)\ge\beta_i
	\right\}
	\right),
\end{equation}
where $z_i=(\alpha_i,\tau_i)\in\R^2$ are distinct points.  This is the
joint tail function denoted by $\widehat T$ in \cite[Definition~2.4]{LZ25}.
Whenever $T$ appears below, the points $z_i$ and the corresponding
thresholds $\beta_i$ are ordered according to the total order $\prec$.

By \cite[Proposition~2.8]{LZ25}, for $(\widehat x,\widehat t,\widehat\beta)\in\R^3$,
\begin{equation}\label{eq:LZ-shift-T}
\begin{aligned}
&T(\beta_1,\ldots,\beta_m;z_1,\ldots,z_m)
\\
&\quad=e^{2\widehat t/3-2\widehat\beta}
T(\beta_1-\widehat\beta,\ldots,\beta_m-\widehat\beta;
  z_1-(\widehat x,\widehat t),\ldots,z_m-(\widehat x,\widehat t)).
\end{aligned}
\end{equation}
By \cite[Proposition~2.12]{LZ25}, for distinct
$z_1,\ldots,z_m\ne(0,0)$,
\begin{equation}\label{eq:LZ-grounded-tail}
\Pr\left(\bigcap_{i=1}^m\left\{\HUTz(z_i)\ge a_i\right\}\right)
=-\frac12\left.
\frac{\partial}{\partial b}
T(\ldots,b,\ldots;\ldots,(0,0),\ldots)
\right|_{b=0},
\end{equation}
where the pairs $(a_i,z_i)$ together with $(b,(0,0))$ are inserted in
$\prec$-order in the arguments of $T$.

The following elementary lemma isolates the differentiation with respect to
the exponential variable that appears after the skew shift.  The random
variables $U_1,\ldots,U_n$ represent the values of the shifted positive-time
field at the observation points, while $V$ represents its value at the
shifted distinguished point. 

\begin{lemma}[Exponential differentiation]\label{lem:exp-diff}
Let $(\Omega,\mathcal F,\mathbb P)$ be a probability space carrying
random variables $U_1,\ldots,U_n,V$ and an independent random variable
$E\sim\operatorname{Exp}(2)$.  For $a_1,\ldots,a_n,s\in\R$, set
\[
C=\max\left\{0,a_1-U_1,\ldots,a_n-U_n\right\}
\]
and
\[
F(s)=\mathbb P\left(E+U_i\ge a_i,\ 1\le i\le n,\ E+V\ge s\right).
\]
If $\mathbb P(V+C=s)=0$, then
\begin{equation}\label{eq:exp-diff}
-\frac12F'(s)
=e^{-2s}\mathbb E\left[
 e^{2V}\1_{\left\{V<s\right\}}
 \prod_{i=1}^n\1_{\left\{U_i-V>a_i-s\right\}}
\right].
\end{equation}
\end{lemma}

\begin{proof}
Conditioning on $(U_1,\ldots,U_n,V)$ gives
\begin{equation}\label{eq:exp-diff-F}
F(s)=\mathbb E\left[e^{-2\max\left\{C,s-V\right\}}\right].
\end{equation}
For fixed $(U_1,\ldots,U_n,V)$, put
$\phi(r)=e^{-2\max\left\{C,r-V\right\}}$.  Since $C\ge0$,
$|\phi'(r)|\le2$ wherever the derivative exists, and consequently
\begin{equation}\label{eq:exp-diff-quotient-bound}
\left|\frac{\phi(s+h)-\phi(s)}{h}\right|\le2,
\qquad h\ne0.
\end{equation}
On $\left\{V+C\ne s\right\}$,
\begin{equation}\label{eq:exp-diff-quotient-limit}
\lim_{h\to0}\frac{\phi(s+h)-\phi(s)}{h}
=-2e^{-2(s-V)}\1_{\left\{V+C<s\right\}}.
\end{equation}
Since $\mathbb P(V+C=s)=0$, dominated convergence applied to
\eqref{eq:exp-diff-F}--\eqref{eq:exp-diff-quotient-limit} yields
\begin{equation}\label{eq:exp-diff-intermediate}
-\frac12F'(s)
=e^{-2s}\mathbb E\left[e^{2V}\1_{\left\{V+C<s\right\}}\right].
\end{equation}
Finally,
\[
\left\{V+C<s\right\}
=
\left\{V<s\right\}
\cap
\bigcap_{i=1}^n\left\{U_i-V>a_i-s\right\}.
\]
Substituting this identity into \eqref{eq:exp-diff-intermediate} proves
\eqref{eq:exp-diff}.
\end{proof}

\begin{lemma}[Skew-shift transfer]\label{lem:skew-transfer}
With the notation of \cref{thm:alltime}, for all $a_1,\ldots,a_m\in\R$,
\begin{equation}\label{eq:alltime-tail-identity}
  \Pr\left(\bigcap_{i=1}^m\left\{\HUTz(z_i)\ge a_i\right\}\right)
  =e^{2t_k/3}\E_{\Pp}\left[
    e^{2V_k}
    \prod_{i=1}^m
    \1_{\{\KPZ(x_i-x_k,t_i-t_k;h_\vee)-V_k\ge a_i\}}
  \right].
\end{equation}
\end{lemma}

\begin{proof}
Set $\widetilde z_i=(x_i-x_k,t_i-t_k)$ and $z^\star=(-x_k,s_k)$.
In \eqref{eq:LZ-grounded-tail}, insert $(b,(0,0))$ in its
$\prec$-position and apply \eqref{eq:LZ-shift-T} with
$(\widehat x,\widehat t,\widehat\beta)=(x_k,t_k,a_k)$.  The shifted
list consists of the pairs
\[
(a_i-a_k,\widetilde z_i),\qquad 1\le i\le m,
\]
together with $(b-a_k,z^\star)$, rearranged in $\prec$-order.  Since
$\widetilde z_k=(0,0)$ and $t_i-t_k\ge0$ for every $i$, the
positive-time representation and
\cite[Proposition~1.5(a),(c)]{LZ25} identify the corresponding tail
probability.  Denote this shifted tail function by $T_{\rm sh}(b)$.  Then
\begin{equation}\label{eq:shifted-tail-probability}
T_{\rm shift}(b)
=
\Pp\left(
 E+\KPZ(x_i-x_k,t_i-t_k;h_\vee)\ge a_i-a_k
 \ \text{for } i\ne k,
 \ E+V_k\ge b-a_k
\right),
\end{equation}
where $E\sim\operatorname{Exp}(2)$ is independent of
$(B_{\mathrm{ts}},\cL)$.  The condition at $\widetilde z_k=(0,0)$ is
$E\ge0$ and is therefore automatic.

Apply \cref{lem:exp-diff} to \eqref{eq:shifted-tail-probability} with
\(U_i=\KPZ(x_i-x_k,t_i-t_k;h_\vee),\qquad
V=V_k,\)
and with thresholds $a_i-a_k$.  Since differentiation in $b$ is the
same as differentiation in $s=b-a_k$, evaluating at $b=0$ means
$s=-a_k$.  By \cite[Remark~2.11]{LZ25}, the boundary event in
\cref{lem:exp-diff} has probability zero, and hence
\begin{align}
-\frac12 T_{\rm shfit}'(0)=
 e^{2a_k}\E_{\Pp}\left[
 e^{2V_k}\1_{\{V_k<-a_k\}}
 \prod_{i\ne k}
 \1_{\{\KPZ(x_i-x_k,t_i-t_k;h_\vee)-V_k>a_i\}}
 \right].
\label{eq:shifted-tail-derivative}
\end{align}
Multiplying \eqref{eq:shifted-tail-derivative} by the shift factor
$e^{2t_k/3-2a_k}$ gives
\[
e^{2t_k/3}\E_{\Pp}\left[
 e^{2V_k}
 \prod_{i=1}^m
 \1_{\{\KPZ(x_i-x_k,t_i-t_k;h_\vee)-V_k>a_i\}}
\right],
\]
because the factor with $i=k$ is precisely
$\1_{\{-V_k>a_k\}}$.  Replacing strict inequalities by non-strict
ones using \cite[Remark~2.11]{LZ25} proves
\eqref{eq:alltime-tail-identity}.
\end{proof}

\begin{proof}[Proof of \cref{thm:alltime}]
	We first determine the normalization of the exponential weight.  Let
	$a_1,\ldots,a_m\downarrow-\infty$ in
	\eqref{eq:alltime-tail-identity}, the monotone convergence theorem yields
	\[
	1
	=
	e^{2t_k/3}
	\E_{\Pp}\left[e^{2V_k}\right].
	\]
	Recalling that $s_k=-t_k$, we obtain
	\begin{equation}\label{eq:exp-moment-proof}
		\E_{\Pp}\left[e^{2V_k}\right]
		=
		e^{2s_k/3}.
	\end{equation}
	Consequently,
	\[
	\E_{\Pp}\left[
	\exp\left\{
	2V_k-\frac23s_k
	\right\}
	\right]
	=
	e^{-2s_k/3}
	\E_{\Pp}\left[e^{2V_k}\right]
	=
	1.
	\]
	Thus \eqref{eq:Qk} defines a probability measure $\mathbb Q_k$
	absolutely continuous with respect to $\Pp$.
	
	Returning to \eqref{eq:alltime-tail-identity} and using
	$t_k=-s_k$, we have, for arbitrary $a_1,\ldots,a_m\in\R$,
	\begin{align}
		&\Pr\left(
		\bigcap_{i=1}^m
		\left\{\HUTz(z_i)\ge a_i\right\}
		\right)
		\nonumber\\
		&\qquad=
		\E_{\Pp}\left[
		\exp\left\{
		2V_k-\frac23s_k
		\right\}
		\prod_{i=1}^m
		\1_{\left\{
			\KPZ(x_i-x_k,t_i-t_k;h_\vee)-V_k\ge a_i
			\right\}}
		\right]
		\nonumber\\
		&\qquad=
		\mathbb Q_k\left(
		\bigcap_{i=1}^m
		\left\{
		\KPZ(x_i-x_k,t_i-t_k;h_\vee)-V_k\ge a_i
		\right\}
		\right).
		\label{eq:alltime-tail-under-Q}
	\end{align}
	Hence the two random vectors in
	\eqref{eq:alltime-representation} have identical joint upper-tail
	probabilities for every $(a_1,\ldots,a_m)\in\R^m$.  Since upper
	orthants
	\(	\prod_{i=1}^m[a_i,\infty),
	(a_1,\ldots,a_m)\in\R^m,\)
	form a measure-determining class on $\R^m$, the two random vectors have
	the same law.  This proves \eqref{eq:alltime-representation}.
\end{proof}

\section{Negative-time large-scale asymptotics}\label{sec:negative-asymptotics}

We prove \cref{prop:negative-longtime} from \cref{thm:alltime}.  This gives a probabilistic proof of \cite[Proposition~1.8(a)]{LZ25} without using the asymptotic analysis of the exact multipoint formula.  The only asymptotic information from integrable probability used below is the one-point GUE Tracy--Widom density estimate \eqref{eq:upper-tail-GUE}. The spatial input is the local Brownian comparison of \cite[Proposition~3.5]{GHZ25}, together with the tent estimate \cite[Corollary~2.21]{GHZ25} and the global parabolic bound \cite[Corollary~10.7]{DOV22}.

The exponential change of measure in the negative-time representation contains a contribution from the Brownian initial profile $h_\vee$ and a contribution from the directed landscape. We first isolate the Brownian part. For a fixed spatial point $y$, weighting the law of $h_\vee$ by $e^{2h_\vee(y)}$ changes the drift of the surrounding Brownian path but does not change its covariance. After subtracting the resulting deterministic tent profile, the increments are again those of a two-sided Brownian motion. The following lemma records this and will be used to identify the spatial fluctuations under the tilted measure. 

Write $B=B_{\mathrm{ts}}$, since \(h_\vee(y)=B(2y)-2|y|\) and $B(2y)\sim N(0,2|y|)$,
\[
\E_{\Pp}\left[e^{2h_\vee(y)}\right]
=
e^{-4|y|}
\E_{\Pp}\left[e^{2B(2y)}\right]
=
e^{-4|y|}e^{4|y|}
=
1.
\]
Hence
\begin{equation}\label{eq:Brownian-tilted-measure}
	\frac{d\Pp_y^B}{d\Pp}
	=
	e^{2h_\vee(y)}
\end{equation}
defines a probability measure on the Brownian environment. 

\begin{lemma}[Brownian exponential tilt]\label{lem:neg-Brownian-tilt}
	Under
	$\Pp_y^B$, the process
	\[
	u\longmapsto
	h_\vee(y+u)-h_\vee(y)+2|u|,
	\qquad u\in\R,
	\]
	is a two-sided Brownian motion with variance parameter $2$.
\end{lemma}

\begin{proof}
	By the Cameron--Martin theorem, see
	\cite[Chapter~VIII, Section~1]{RY99}, the exponential tilt
	$e^{2B(2y)-4|y|}$ preserves the covariance of $B$ and changes its mean to
	\[
	\E_{\Pp_y^B}[B(t)]
	=
	2\Cov\left(B(t),B(2y)\right).
	\]
	Set
	\[
	X_y(u)
	=
	h_\vee(y+u)-h_\vee(y)+2|u|.
	\]
	Then
	\[
	X_y(u)
	=
	B(2y+2u)-B(2y)
	-2\left(|y+u|-|y|\right)
	+2|u|.
	\]
	
	We verify that $X_y$ is centered.  By symmetry, it is enough to take
	$y\ge0$.  The Cameron--Martin shift gives
	\[
	\E_{\Pp_y^B}
	\left[
	B(2y+2u)-B(2y)
	\right]
	=
	\begin{cases}
		0, & u\ge0,\\
		4u, & -y\le u<0,\\
		-4y, & u<-y.
	\end{cases}
	\]
	On the other hand,
	\[
	-2\left(|y+u|-|y|\right)+2|u|
	=
	\begin{cases}
		0, & u\ge0,\\
		-4u, & -y\le u<0,\\
		4y, & u<-y.
	\end{cases}
	\]
	Hence \(\E_{\Pp_y^B}[X_y(u)]=0, u\in\R.\)
	
	The exponential tilt does not change covariances.  Therefore
	\[
	\Cov_{\Pp_y^B}
	\left(
	X_y(u),X_y(v)
	\right)
	=
	\begin{cases}
		2\min\left\{|u|,|v|\right\}, & uv\ge0,\\
		0, & uv<0.
	\end{cases}
	\]
	Thus $X_y$ is a centered Gaussian process whose restrictions to
	$[0,\infty)$ and $(-\infty,0]$ are independent Brownian motions with
	variance parameter $2$.  This proves the claim.
\end{proof}

The second input concerns the one-point directed-landscape
height under the exponential change of measure.  By
\eqref{eq:landscape-one-point},
\(\cL(y,0;0,T)
\overset d=
T^{1/3}\chi-\frac{y^2}{T},
\chi\sim F_{\rm GUE}.\)
Consequently, the exponential weight
$e^{2\cL(y,0;0,T)}$ separates into the deterministic factor
$e^{-2y^2/T}$ and the factor $e^{2T^{1/3}\chi}$ depending on the
one-point height.  The following lemma identifies the fluctuations of
$\chi$ under this exponential tilt.  It will be used in the proof of
\cref{prop:negative-longtime} to obtain the Gaussian fluctuation of the
directed-landscape height on the scale $\sqrt T$.

\begin{lemma}[Exponentially tilted GUE law]\label{lem:neg-tilted-GUE}
	For $T>0$, define a probability measure $\mu_T$ on $\R$ by
	\begin{equation}\label{eq:GUE-tilted-measure}
		\mu_T(du)
		=
		\frac{
			e^{2T^{1/3}u}\,dF_{\rm GUE}(u)
		}{
			\displaystyle
			\int_{\R}e^{2T^{1/3}v}\,dF_{\rm GUE}(v)
		}.
	\end{equation}
	If $\chi_T$ has distribution $\mu_T$, then
	\begin{equation}\label{eq:neg-tilted-GUE}
		\frac{T^{1/3}\chi_T-T}{\sqrt T}
		\Longrightarrow
		N(0,1)
	\end{equation}
	in distribution as $T\to\infty$.
\end{lemma}

\begin{proof}
	Write $f_{\rm GUE}=F_{\rm GUE}'$ and let 
	\(	Z_T
	=
	\int_{\R}
	e^{2T^{1/3}u}f_{\rm GUE}(u)\,du\)
	be the normalizing constant in \eqref{eq:GUE-tilted-measure}.  Set \(	u=T^{2/3}+T^{1/6}v.\)
	For $v$ in a fixed compact set, $u\to\infty$ uniformly as
	$T\to\infty$.  By \eqref{eq:upper-tail-GUE},
	\[
	f_{\rm GUE}(u)
	=
	\frac{1+O(u^{-3/2})}{8\pi u}
	\exp\left\{
	-\frac43u^{3/2}
	\right\}.
	\]

	A Taylor expansion at $1$ gives
	\begin{equation}\label{eq:laplace-local}
		2T^{1/3}u-\frac43u^{3/2}
		=
		\frac23T-\frac12v^2
		+O\left(T^{-1/2}|v|^3\right),
	\end{equation}
	locally uniformly in $v$.  Moreover, \(	\frac{T^{1/6}}{u}
	=
	T^{-1/2}
	\left(1+O(T^{-1/2}|v|)\right).\)
	Consequently, for every fixed $M<\infty$,
	\begin{equation}\label{eq:GUE-local-density}
		e^{2T^{1/3}u}
		f_{\rm GUE}(u)T^{1/6}
		=
		\frac{e^{2T/3}}{8\pi T^{1/2}}
		e^{-v^2/2}
		\left(1+o(1)\right)
	\end{equation}
	uniformly for $|v|\le M$.
	
	It remains to control the contribution away from the saddle point.
	For $s\ge0$, set \(	\phi(s)=2s-\frac43s^{3/2}.\)
	Then \(	\phi'(s)=2-2s^{1/2},
	\phi''(s)=-s^{-1/2},\)
	so $s=1$ is the unique maximizer of $\phi$ on $[0,\infty)$ and
	$\phi(1)=2/3$.  Hence there exist $c,\eta>0$ such that
	\[
	\phi(s)
	\le
	\frac23-c(s-1)^2
	\qquad
	\text{for }|s-1|\le\eta,
	\]
	and
	\[
	\phi(s)
	\le
	\frac23-c
	\qquad
	\text{for }|s-1|\ge\eta.
	\]
	When $u=T^{2/3}s$, \(	2T^{1/3}u-\frac43u^{3/2}=T\phi(s),v=T^{1/2}(s-1).\)
	Thus, in the region $|s-1|\le\eta$, \(	T\phi(s)\le\frac23T-cv^2.\)
	Using the upper bound for $f_{\rm GUE}$ that follows from
	\eqref{eq:upper-tail-GUE}, we obtain, for every $M\ge1$,
	\begin{equation}\label{eq:GUE-tail-local}
		\int_{\substack{|v|>M\\ |s-1|\le\eta}}
		e^{2T^{1/3}u}f_{\rm GUE}(u)T^{1/6}\,dv
		\le
		C\frac{e^{2T/3}}{T^{1/2}}
		\int_{|v|>M}e^{-cv^2}\,dv.
	\end{equation}
	On the region $|s-1|\ge\eta$ with $u$ sufficiently large, the second bound on $\phi$ gives
	\begin{equation}\label{eq:GUE-tail-far}
		\int_{\substack{|s-1|\ge\eta\\ u\ge u_0}}
		e^{2T^{1/3}u}f_{\rm GUE}(u)\,du
		\le
		Ce^{(2/3-c)T}.
	\end{equation}
	For $u<u_0$, where $u_0$ is fixed,
	\[
	\int_{-\infty}^{u_0}
	e^{2T^{1/3}u}f_{\rm GUE}(u)\,du
	\le
	e^{2T^{1/3}u_0}
	=
	o\left(
	\frac{e^{2T/3}}{T^{1/2}}
	\right).
	\]
	Together with \eqref{eq:GUE-local-density}, these estimates imply
	\begin{equation}\label{eq:GUE-tail-tight}
		\limsup_{T\to\infty}
		\frac{
			\displaystyle
			\int_{|v|>M}
			e^{2T^{1/3}u}f_{\rm GUE}(u)T^{1/6}\,dv
		}{
			Z_T
		}
		\le
		Ce^{-cM^2}.
	\end{equation}
	
	Finally, for every bounded continuous function $\psi$,
	\eqref{eq:GUE-local-density} and \eqref{eq:GUE-tail-tight} give
	\[
	\lim_{T\to\infty}
	\int_{\R}
	\psi\left(
	\frac{T^{1/3}u-T}{\sqrt T}
	\right)
	\mu_T(du)
	=
	\frac{1}{\sqrt{2\pi}}
	\int_{\R}\psi(v)e^{-v^2/2}\,dv.
	\]
	This proves \eqref{eq:neg-tilted-GUE}.
\end{proof}

The next lemma is the only place where the local Brownian structure of the directed landscape is used.  It converts the exponential tilt of the variational maximum into the corresponding exponential weight along a finite directed-landscape chain.

For the remainder of this subsection, fix
\[
R=r_0>r_1>\cdots>r_n>r_{n+1}=0,
\qquad
s_j=\lambda(R-r_j),
\]
and set $y_{n+1}=0$.  Write
\[
\Delta_\ell=s_\ell-s_{\ell-1},
\qquad 1\le\ell\le n+1.
\]
For $y=(y_0,\ldots,y_n)\in\R^{n+1}$, define the chain weight
\begin{equation}\label{eq:chain-weight}
	W_\lambda(y)
	=
	\exp\left\{
	2h_\vee(y_0)
	+
	2\sum_{\ell=1}^{n+1}\cL(y_{\ell-1},s_{\ell-1};y_\ell,s_\ell)
	\right\},
	\qquad
	\mathcal Z_\lambda
	=
	\int_{\R^{n+1}}W_\lambda(y)\,dy.
\end{equation}
We also introduce the probability measure
\begin{equation}\label{eq:chain-tilted-measure}
	d\widehat{\Pp}_\lambda
	=
	\frac{W_\lambda(y)}
	{\E_{\Pp}[\mathcal Z_\lambda]}
	dy\,d\Pp
\end{equation}
on the environment together with the chain $y$.  The Gaussian endpoint
factorization below shows that $0<\E_{\Pp}[\mathcal Z_\lambda]<\infty$,
so \eqref{eq:chain-tilted-measure} is well defined.
Recall that, for $t>0$, \(V_t=\KPZ(0,t;h_\vee).\)  We write
\[
\frac{d\mathbb Q_t}{d\Pp}=\exp\left\{2V_t-\frac23t\right\}.
\]
For the finitely many prescribed $x_{j,a}\in\R$, set
\begin{equation}\label{eq:chain-coordinates}
	X_{j,a}^{(\lambda)}
	=
	\frac{
		\KPZ(\lambda^{1/2}x_{j,a}/\sqrt2,s_j;h_\vee)
		-V_{\lambda R}+\lambda r_j
	}{
		\sqrt{2\lambda}
	},
\end{equation}
and
\begin{equation}\label{eq:chain-replacement-coordinate}
	Z_{j,a}^{(\lambda)}(y)
	=
	\frac{
		\lambda r_j-\sum_{\ell=j+1}^{n+1}\cL(y_{\ell-1},s_{\ell-1};y_\ell,s_\ell)
	}{
		\sqrt{2\lambda}
	}
	-
	\left|
	x_{j,a}-\frac{\sqrt2\,y_j}{\sqrt\lambda}
	\right|.
\end{equation}

\begin{lemma}[Finite-chain replacement]\label{lem:neg-chain-replacement}
	For every bounded Lipschitz function $F$,
	\begin{equation}\label{eq:neg-chain-replacement}
		\E_{\mathbb Q_{\lambda R}}
		\left[
		F\left(
		\left\{
		X_{j,a}^{(\lambda)}
		\right\}_{j,a}
		\right)
		\right]
		=
		\E_{\widehat{\Pp}_\lambda}
		\left[
		F\left(
		\left\{
		Z_{j,a}^{(\lambda)}(y)
		\right\}_{j,a}
		\right)
		\right]
		+o(1)
	\end{equation}
	as $\lambda\to\infty$.
\end{lemma}

\begin{proof}
	Define the normalized residual factor
	\begin{equation}\label{eq:chain-residual-factor}
		\rho_\lambda
		=
		\frac{
			e^{2V_{\lambda R}}/\mathcal Z_\lambda
		}{
			\E_{\widehat{\Pp}_\lambda}
			\left[e^{2V_{\lambda R}}/\mathcal Z_\lambda\right]
		}.
	\end{equation}
	A direct substitution of \eqref{eq:chain-tilted-measure}, together with
	$d\mathbb Q_{\lambda R}/d\Pp
	=e^{2V_{\lambda R}-2\lambda R/3}$, gives
	\begin{equation}\label{eq:Q-chain-exact}
		\E_{\mathbb Q_{\lambda R}}
		\left[
		F\left(
		\left\{
		X_{j,a}^{(\lambda)}
		\right\}_{j,a}
		\right)
		\right]
		=
		\E_{\widehat{\Pp}_\lambda}
		\left[
		\rho_\lambda
		F\left(
		\left\{
		X_{j,a}^{(\lambda)}
		\right\}_{j,a}
		\right)
		\right].
	\end{equation}
	
We next identify the distribution of the chain under
$\widehat{\Pp}_\lambda$.  For fixed
$y=(y_0,\ldots,y_n)$, independence of the Brownian initial profile and
the directed-landscape increments on the disjoint time intervals
$[s_{\ell-1},s_\ell]$ gives
\[
\E_{\Pp}\left[W_\lambda(y)\right]
=
\E_{\Pp}\left[e^{2h_\vee(y_0)}\right]
\prod_{\ell=1}^{n+1}
\E_{\Pp}\left[e^{2\cL(y_{\ell-1},s_{\ell-1};y_\ell,s_\ell)}\right].
\]
The first factor equals $1$ by
\cref{lem:neg-Brownian-tilt}.  By
\eqref{eq:landscape-one-point},
\[
\cL(y_{\ell-1},s_{\ell-1};y_\ell,s_\ell)
\overset d=
\Delta_\ell^{1/3}\chi_\ell
-
\frac{(y_\ell-y_{\ell-1})^2}{\Delta_\ell},
\qquad
\chi_\ell\sim F_{\rm GUE}.
\]
Hence
\[
\E_{\Pp}\left[e^{2\cL(y_{\ell-1},s_{\ell-1};y_\ell,s_\ell)}\right]
=
\exp\left\{
-2\frac{(y_\ell-y_{\ell-1})^2}{\Delta_\ell}
\right\}
\E\left[e^{2\Delta_\ell^{1/3}\chi_\ell}\right],
\]
where the last factor is independent of the endpoints.  It follows from
\eqref{eq:chain-tilted-measure} that the endpoint marginal of
$\widehat{\Pp}_\lambda$ has density proportional to
\begin{equation}\label{eq:chain-endpoint-density}
	\prod_{\ell=1}^{n+1}
	\exp\left\{
	-2\frac{(y_\ell-y_{\ell-1})^2}{\Delta_\ell}
	\right\}.
\end{equation}

Equivalently, under this marginal the increments
$y_\ell-y_{\ell-1}$, $1\le\ell\le n+1$, are independent centered
Gaussian random variables with variances $\Delta_\ell/4$.  Since
$y_{n+1}=0$,
\[
y_j
=
-\sum_{\ell=j+1}^{n+1}
(y_\ell-y_{\ell-1}),
\qquad
\operatorname{Var}_{\widehat{\Pp}_\lambda}(y_j)
=
\frac14\sum_{\ell=j+1}^{n+1}\Delta_\ell
=
\frac{\lambda r_j}{4}.
\]
The number of chain points is fixed, and therefore
\begin{equation}\label{eq:chain-endpoint-tight}
	\lim_{M\to\infty}\limsup_{\lambda\to\infty}
	\widehat{\Pp}_\lambda
	\left(
	\max_{0\le j\le n}
	\frac{|y_j|}{\sqrt\lambda}>M
	\right)
	=0.
\end{equation}

We next consider the passage values along the chain.  Set \(\Gamma_\ell(y)=\cL(y_{\ell-1},s_{\ell-1};y_\ell,s_\ell)+\frac{(y_\ell-y_{\ell-1})^2}{\Delta_\ell}.\)
Conditionally on the endpoints, the independent-increment property of
the directed landscape and the factorization of $W_\lambda$ show that
$\Gamma_1,\ldots,\Gamma_{n+1}$ remain independent under
$\widehat{\Pp}_\lambda$.  Moreover, the conditional law of
$\Gamma_\ell$ is the exponential tilt of
$\Delta_\ell^{1/3}\chi_\ell$ by the weight
$e^{2\Delta_\ell^{1/3}\chi_\ell}$.  Therefore
\cref{lem:neg-tilted-GUE}, applied with $T=\Delta_\ell$, gives
\begin{equation}\label{eq:chain-height-limit}
	\left\{
	\frac{\Gamma_\ell(y)-\Delta_\ell}
	{\sqrt{\Delta_\ell}}
	\right\}_{\ell=1}^{n+1}
	\Longrightarrow
	\left\{
	G_\ell
	\right\}_{\ell=1}^{n+1},
\end{equation}
where $G_1,\ldots,G_{n+1}$ are independent standard Gaussian random
variables.  Since
$\Delta_\ell=\lambda(r_{\ell-1}-r_\ell)$ and the number of slabs is
fixed, in particular
\begin{equation}\label{eq:chain-height-window}
	\Gamma_\ell(y)
	=
	\Delta_\ell
	+
	O_{\widehat{\Pp}_\lambda}(\sqrt\lambda),
	\qquad
	1\le\ell\le n+1.
\end{equation}

For $M<\infty$, set
\[
\mathcal A_{\lambda,M}
=
\left\{
\max_{0\le j\le n}|y_j|\le M\sqrt\lambda,\quad
\max_{1\le\ell\le n+1}|\Gamma_\ell-\Delta_\ell|
\le M\sqrt\lambda
\right\}.
\]
By \eqref{eq:chain-endpoint-tight} and
\eqref{eq:chain-height-window},
\begin{equation}\label{eq:chain-good-event}
\lim_{M\to\infty}\limsup_{\lambda\to\infty}
\widehat{\Pp}_\lambda(\mathcal A_{\lambda,M}^c)=0.
\end{equation}
Fix $M$.  On $\mathcal A_{\lambda,M}$ one has
$\Delta_\ell\asymp\lambda$,
$|y_\ell-y_{\ell-1}|=O(\sqrt\lambda)$, and
\(\Delta_\ell^{-1/3}\Gamma_\ell
=
\Delta_\ell^{2/3}+O(\lambda^{1/6}).\)
After translating the slab endpoints, applying the shift invariance of
\cite[Theorem~3.3]{GHZ25}, and then using $1{:}2{:}3$ scaling, each
slab is therefore in the deep upper-tail regime of \cite{GHZ25},
uniformly on $\mathcal A_{\lambda,M}$.

The initial Brownian profile must be treated separately from the
landscape slabs.  Conditional on $y_0$, its law under
$\widehat{\Pp}_\lambda$ is tilted by $e^{2h_\vee(y_0)}$.  Hence
\cref{lem:neg-Brownian-tilt} gives
\begin{equation}\label{eq:chain-initial-tent}
 h_\vee(y_0+u)-h_\vee(y_0)
 =-2|u|+\widetilde B_0(u),
\end{equation}
where $\widetilde B_0$ is a two-sided Brownian motion with variance
parameter $2$.  In particular, with
$\ell_\lambda=(\log\lambda)^2$,
\begin{equation}\label{eq:chain-initial-local}
 \sup_{|u|\le\ell_\lambda}|\widetilde B_0(u)|
 =O_{\widehat{\Pp}_\lambda}(\log\lambda)
 =o_{\widehat{\Pp}_\lambda}(\sqrt\lambda).
\end{equation}

For the landscape slabs, Proposition~3.5 and Corollary~2.21 of
\cite{GHZ25} give, uniformly on $\mathcal A_{\lambda,M}$, the
corresponding two-sided tent approximation.  More precisely, whenever
$|u|,|v|\le\ell_\lambda$,
\begin{equation}\label{eq:chain-slab-local}
\cL(y_{\ell-1}+u,s_{\ell-1};y_\ell+v,s_\ell)
-
\cL(y_{\ell-1},s_{\ell-1};y_\ell,s_\ell)
+
2|u|+2|v|
=
o_{\widehat{\Pp}_\lambda}(\sqrt\lambda)
\end{equation}
uniformly over $1\le\ell\le n+1$.  In the Brownian comparison model, the fluctuation in
\eqref{eq:chain-initial-tent} and the bridge fluctuations associated with
\eqref{eq:chain-slab-local} are independent across the disjoint slabs.

Equations \eqref{eq:chain-initial-tent}--\eqref{eq:chain-slab-local},
together with metric composition, imply the following two estimates.
For
$\xi_{j,a}=\lambda^{1/2}x_{j,a}/\sqrt2$ and every $\eps>0$,
\begin{align}
&\widehat{\Pp}_\lambda\left(
\mathcal A_{\lambda,M}\cap
\left\{
\max_{j,a}
\left|
\KPZ(\xi_{j,a},s_j;h_\vee)
-h_\vee(y_0)-\sum_{\ell=1}^j\cL(y_{\ell-1},s_{\ell-1};y_\ell,s_\ell)
+2|\xi_{j,a}-y_j|
\right|>\eps\sqrt\lambda
\right\}
\right)
\rightarrow0,
\label{eq:chain-profile-tent}
\\
&\widehat{\Pp}_\lambda\left(
\mathcal A_{\lambda,M}\cap
\left\{
V_{\lambda R}
-h_\vee(y_0)-\sum_{\ell=1}^{n+1}\cL(y_{\ell-1},s_{\ell-1};y_\ell,s_\ell)
>\eps\sqrt\lambda
\right\}
\right)
\rightarrow0.
\label{eq:neg-max-gap}
\end{align}
Indeed, first restrict the intermediate variational points to
$|z_j-y_j|\le\ell_\lambda$.  On this set,
\eqref{eq:chain-initial-local} and \eqref{eq:chain-slab-local} show
that the value of the variational functional differs from the chain
value by the deterministic tent loss, up to
$o_{\widehat{\Pp}_\lambda}(\sqrt\lambda)$.  Taking the supremum gives
the compact-window parts of \eqref{eq:chain-profile-tent} and
\eqref{eq:neg-max-gap}.  If one of the intermediate points leaves this
window, \eqref{eq:chain-initial-tent} for the initial profile and
\cite[Corollary~2.21]{GHZ25} for the landscape slabs give a linear loss
that dominates the Brownian fluctuations up to the edge of the
upper-tail window.  Beyond that window the global parabolic estimate
\cite[Corollary~10.7]{DOV22} gives quadratic decay.  The probability
that a point outside the localized region improves the chain value by
$\eps\sqrt\lambda$ therefore tends to zero.  This proves
\eqref{eq:chain-profile-tent} and \eqref{eq:neg-max-gap}.

Using the definitions of $X_{j,a}^{(\lambda)}$ and
$Z_{j,a}^{(\lambda)}(y)$, we have the exact identity
\begin{align}
X_{j,a}^{(\lambda)}-Z_{j,a}^{(\lambda)}(y)
&=
\frac{
\KPZ(\xi_{j,a},s_j;h_\vee)
-h_\vee(y_0)-\sum_{\ell=1}^j\cL(y_{\ell-1},s_{\ell-1};y_\ell,s_\ell)
+2|\xi_{j,a}-y_j|
}{\sqrt{2\lambda}}
\nonumber\\
&\quad-
\frac{
V_{\lambda R}
-h_\vee(y_0)-\sum_{\ell=1}^{n+1}\cL(y_{\ell-1},s_{\ell-1};y_\ell,s_\ell)
}{\sqrt{2\lambda}}.
\label{eq:chain-coordinate-identity}
\end{align}
By \eqref{eq:chain-good-event}--\eqref{eq:neg-max-gap},
\begin{equation}\label{eq:chain-approximation}
\max_{j,a}
\left|X_{j,a}^{(\lambda)}-Z_{j,a}^{(\lambda)}(y)\right|
\longrightarrow0
\end{equation}
in $\widehat{\Pp}_\lambda$-probability.

It remains to remove the factor $\rho_\lambda$.  Let
\[
\mathcal G_\lambda
=
\sigma\left(
\left\{
 y_\ell-y_{\ell-1},\,\Gamma_\ell(y)
\right\}_{\ell=1}^{n+1}
\right).
\]
Since $y_{n+1}=0$,
\[
y_j=-\sum_{\ell=j+1}^{n+1}(y_\ell-y_{\ell-1}),
\qquad
\cL(y_{\ell-1},s_{\ell-1};y_\ell,s_\ell)
=
\Gamma_\ell-\frac{(y_\ell-y_{\ell-1})^2}{\Delta_\ell},
\]
so $\{Z_{j,a}^{(\lambda)}(y)\}_{j,a}$ is
$\mathcal G_\lambda$-measurable.

We first prove the conditional decoupling
\begin{equation}\label{eq:rho-conditional}
\left\|
\E_{\widehat{\Pp}_\lambda}
\left[\rho_\lambda\mid\mathcal G_\lambda\right]-1
\right\|_{L^1(\widehat{\Pp}_\lambda)}
\longrightarrow0.
\end{equation}
Fix $M,K<\infty$ and work on $\mathcal A_{\lambda,M}$.  Conditional on
$\mathcal G_\lambda$, the Brownian initial environment and the directed
landscape on the $n+1$ disjoint slabs remain independent.  The former
has the tilted law of \cref{lem:neg-Brownian-tilt}.  On the $\ell$th
slab, conditioning on $\Gamma_\ell$ removes the exponential factor
$e^{2\cL(y_{\ell-1},s_{\ell-1};y_\ell,s_\ell)}$, so the remaining conditional law is the original
landscape law conditioned on its centered passage height.

To apply \cite[Proposition~3.5]{GHZ25} under this exact conditioning,
first condition on intervals of width $\delta>0$ containing the
finitely many values $\Gamma_\ell$.  The comparison in
\cite[Proposition~3.5]{GHZ25} is uniform over the height windows in
$\mathcal A_{\lambda,M}$, and in the comparison model the Brownian
bridges are independent of the corresponding endpoint data.  Since the
endpoint increments have Gaussian densities and the variables
$\Gamma_\ell$ have continuous tilted-GUE densities, Lebesgue
differentiation lets $\delta\downarrow0$ and yields the same comparison
for a regular conditional law given $\mathcal G_\lambda$.
Consequently, after restricting the maximum and the integral in
$\rho_\lambda$ to the $\ell_\lambda$-neighborhoods used above, the
conditional law of the centered local profiles is, in total variation,
$o(1)$ from a product of the Brownian process in
\eqref{eq:chain-initial-tent} and independent rate-$2$ Brownian
bridges.  The affine slopes converge uniformly on
$\mathcal A_{\lambda,M}$ to the deterministic tent slopes by
\cite[Corollary~2.21]{GHZ25}.

Since $\rho_\lambda\wedge K$ is bounded, this conditional total-variation
comparison gives a deterministic number $c_{\lambda,K}$ such that
\begin{equation}\label{eq:rho-brownian-comparison}
\E_{\widehat{\Pp}_\lambda}\left[
\left|
\E_{\widehat{\Pp}_\lambda}
\left[\rho_\lambda\wedge K\mid\mathcal G_\lambda\right]
-c_{\lambda,K}
\right|
\1_{\mathcal A_{\lambda,M}}
\right]
\longrightarrow0.
\end{equation}
The localization error is $o_{L^1}(1)$ by
\eqref{eq:chain-initial-tent}, \cite[Corollary~2.21]{GHZ25}, and
\cite[Corollary~10.7]{DOV22}.  Taking expectations in
\eqref{eq:rho-brownian-comparison}, then using
\eqref{eq:chain-good-event} and letting $M\to\infty$, gives
\begin{equation}\label{eq:rho-truncated}
\left\|
\E_{\widehat{\Pp}_\lambda}
\left[\rho_\lambda\wedge K\mid\mathcal G_\lambda\right]
-
\E_{\widehat{\Pp}_\lambda}[\rho_\lambda\wedge K]
\right\|_1
\longrightarrow0.
\end{equation}

It remains to remove the truncation.  The normalization in
\eqref{eq:chain-residual-factor} can be controlled explicitly.  From
\eqref{eq:chain-tilted-measure},
\begin{equation}\label{eq:ratio-normalization-exact}
\E_{\widehat{\Pp}_\lambda}
\left[\frac{e^{2V_{\lambda R}}}{\mathcal Z_\lambda}\right]
=
\frac{\E_{\Pp}[e^{2V_{\lambda R}}]}
{\E_{\Pp}[\mathcal Z_\lambda]}.
\end{equation}
The numerator equals $e^{2\lambda R/3}$ by
\eqref{eq:exp-moment-proof}.  For the denominator, the endpoint
factorization above and the change of variables
$u_\ell=y_\ell-y_{\ell-1}$ give
\begin{equation}\label{eq:chain-Z-mean}
\E_{\Pp}[\mathcal Z_\lambda]
=
\prod_{\ell=1}^{n+1}
\left(
\sqrt{\frac{\pi\Delta_\ell}{2}}\,
\E\left[e^{2\Delta_\ell^{1/3}\chi}\right]
\right).
\end{equation}
The Laplace calculation in the proof of
\cref{lem:neg-tilted-GUE} gives
\[
\E\left[e^{2T^{1/3}\chi}\right]
=
\frac{e^{2T/3}}{4\sqrt{2\pi T}}
\left(1+o(1)\right),
\qquad T\to\infty.
\]
Since $\sum_{\ell=1}^{n+1}\Delta_\ell=\lambda R$,
\begin{equation}\label{eq:ratio-normalization}
\E_{\Pp}[\mathcal Z_\lambda]
=
8^{-(n+1)}e^{2\lambda R/3}(1+o(1)),
\qquad
\E_{\widehat{\Pp}_\lambda}
\left[\frac{e^{2V_{\lambda R}}}{\mathcal Z_\lambda}\right]
=
8^{n+1}(1+o(1)).
\end{equation}
In particular, the denominator in \eqref{eq:chain-residual-factor} is
bounded above and away from zero for all sufficiently large $\lambda$.

Let $z^*$ be a maximizer of the variational functional
\[
\Phi_\lambda(z)
=
h_\vee(z_0)+\sum_{\ell=1}^{n+1}
\cL(z_{\ell-1},s_{\ell-1};z_\ell,s_\ell),
\qquad z_{n+1}=0.
\]
Since $\mathcal Z_\lambda$ contains the integral over the unit cube
$Q(z^*)$ centered at $z^*$,
\begin{equation}\label{eq:ratio-cube-bound}
\frac{e^{2V_{\lambda R}}}{\mathcal Z_\lambda}
\le
\exp\left\{
2\sup_{z,z'\in Q(z^*)}
|\Phi_\lambda(z)-\Phi_\lambda(z')|
\right\}.
\end{equation}
On the localized event from
\eqref{eq:chain-profile-tent}--\eqref{eq:neg-max-gap}, the oscillation
in \eqref{eq:ratio-cube-bound} is a finite sum of unit-scale
oscillations of the tilted Brownian initial profile and of the one-sided
slab profiles.  The Cameron--Martin description in
\cref{lem:neg-Brownian-tilt} and the Brownian comparison in
\cite[Proposition~3.5]{GHZ25} give Gaussian exponential moments for
these oscillations, uniformly over the coarse variables.  Integrating
over the Gaussian endpoint density \eqref{eq:chain-endpoint-density}
and the Gaussian tail bound implicit in
\eqref{eq:GUE-tail-tight} removes the restriction to
$\mathcal A_{\lambda,M}$.  Together with the global parabolic estimate
outside the upper-tail windows, this yields some $q>1$ such that
\begin{equation}\label{eq:local-ratio-ui}
\sup_{\lambda\ge\lambda_0}
\E_{\widehat{\Pp}_\lambda}
\left[
\left(
\frac{e^{2V_{\lambda R}}}{\mathcal Z_\lambda}
\right)^q
\right]
<\infty.
\end{equation}
By \eqref{eq:ratio-normalization}, the same bound holds for
$\rho_\lambda^q$.  Since
$\E_{\widehat{\Pp}_\lambda}[\rho_\lambda]=1$,
\[
\E_{\widehat{\Pp}_\lambda}
\left[\rho_\lambda\1_{\{\rho_\lambda>K\}}\right]
\le
K^{1-q}\E_{\widehat{\Pp}_\lambda}[\rho_\lambda^q].
\]
Combining this estimate with \eqref{eq:rho-truncated}, then letting
$\lambda\to\infty$ and $K\to\infty$, proves
\eqref{eq:rho-conditional}.

By \eqref{eq:chain-approximation} and boundedness of $F$,
\[
F\left(\{X_{j,a}^{(\lambda)}\}_{j,a}\right)
-
F\left(\{Z_{j,a}^{(\lambda)}(y)\}_{j,a}\right)
\longrightarrow0
\]
in $L^{q/(q-1)}(\widehat{\Pp}_\lambda)$.  H\"older's inequality and
\eqref{eq:local-ratio-ui} therefore give
\begin{equation}\label{eq:weighted-coordinate-replacement}
\E_{\widehat{\Pp}_\lambda}
\left[
\rho_\lambda
\left|
F\left(\{X_{j,a}^{(\lambda)}\}_{j,a}\right)
-
F\left(\{Z_{j,a}^{(\lambda)}(y)\}_{j,a}\right)
\right|
\right]
\longrightarrow0.
\end{equation}
Finally, $F(\{Z_{j,a}^{(\lambda)}(y)\}_{j,a})$ is
$\mathcal G_\lambda$-measurable, and hence
\begin{align*}
\left|
\E_{\widehat{\Pp}_\lambda}
\left[
\rho_\lambda F\left(\{Z_{j,a}^{(\lambda)}(y)\}_{j,a}\right)
\right]
-
\E_{\widehat{\Pp}_\lambda}
\left[
F\left(\{Z_{j,a}^{(\lambda)}(y)\}_{j,a}\right)
\right]
\right|
\le
\|F\|_\infty
\left\|
\E_{\widehat{\Pp}_\lambda}
\left[\rho_\lambda\mid\mathcal G_\lambda\right]-1
\right\|_1
\longrightarrow0.
\end{align*}
Together with \eqref{eq:Q-chain-exact} and
\eqref{eq:weighted-coordinate-replacement}, this proves
\eqref{eq:neg-chain-replacement}.

\end{proof}

\begin{proof}[Proof of \cref{prop:negative-longtime}]
Fix finitely many points in $\R\times(-\infty,0)$.  Choose
$R>r_1>\cdots>r_n>0$ so that their distinct time coordinates are
$-r_1,\ldots,-r_n$, and adjoin the earlier point $(0,-\lambda R)$.
By \cref{thm:alltime}, the rescaled joint law of the original points is
the law under $\mathbb Q_{\lambda R}$ of
\begin{equation}\label{eq:negative-Q-vector}
\left\{
\frac{
\KPZ(\lambda^{1/2}x_{j,a}/\sqrt2,\lambda(R-r_j);h_\vee)
-V_{\lambda R}+\lambda r_j
}{\sqrt{2\lambda}}
\right\}_{j,a}.
\end{equation}
Apply \cref{lem:neg-chain-replacement}.  From
\eqref{eq:chain-endpoint-density}, the increments
$y_\ell-y_{\ell-1}$ are independent centered Gaussians with variances
$\Delta_\ell/4$.  Since $y_{n+1}=0$,
\begin{equation}\label{eq:space-gaussian-limit}
\left\{
\frac{\sqrt2\,y_j}{\sqrt\lambda}
\right\}_{j=1}^n
\Longrightarrow
\left\{G_-(r_j)\right\}_{j=1}^n,
\end{equation}
where $G_-$ is a centered Gaussian process with
\[
\operatorname{Cov}\left(G_-(r),G_-(s)\right)
=\frac12\min\{r,s\}.
\]

Conditionally on the endpoints, the variables $\Gamma_\ell$ are
independent of the endpoint increments and have the tilted GUE laws
identified in \cref{lem:neg-tilted-GUE}.  Hence
\[
\left\{
\frac{\Delta_\ell-\Gamma_\ell}{\sqrt{2\lambda}}
\right\}_{\ell=1}^{n+1}
\Longrightarrow
\left\{N_\ell\right\}_{\ell=1}^{n+1},
\]
where the $N_\ell$ are independent centered Gaussians with
\(\operatorname{Var}(N_\ell)
=\frac{r_{\ell-1}-r_\ell}{2},\)
and this convergence is joint with \eqref{eq:space-gaussian-limit}.
Moreover,
\(\frac{(y_\ell-y_{\ell-1})^2}{\Delta_\ell}
=O_{\widehat{\Pp}_\lambda}(1),\)
so the parabolic correction is negligible after division by
$\sqrt\lambda$.  Therefore
\begin{equation}\label{eq:height-gaussian-limit}
\left\{
\frac{
\lambda r_j-
\sum_{\ell=j+1}^{n+1}\cL(y_{\ell-1},s_{\ell-1};y_\ell,s_\ell)
}{\sqrt{2\lambda}}
\right\}_{j=1}^n
\Longrightarrow
\left\{G_+(r_j)\right\}_{j=1}^n,
\end{equation}
where $G_+$ is independent of $G_-$ and
\(\Cov\left(G_+(r),G_+(s)\right)
=\frac12\min\{r,s\}.\)

Define \(B_1(r)=G_+(r)-G_-(r),
B_2(r)=G_+(r)+G_-(r).\)
Then $B_1$ and $B_2$ are independent standard Brownian motions.  
Using \eqref{eq:chain-replacement-coordinate},
\eqref{eq:space-gaussian-limit}, and \eqref{eq:height-gaussian-limit},
the $(j,a)$ coordinate converges to
\[
G_+(r_j)-|x_{j,a}-G_-(r_j)|
=
\min\left\{
B_1(r_j)+x_{j,a},
B_2(r_j)-x_{j,a}
\right\}.
\]
Together with \eqref{eq:negative-Q-vector} and
\cref{lem:neg-chain-replacement}, this proves
\eqref{eq:negative-longtime}.
\end{proof}


\begin{thebibliography}{99}

\bibitem{BL24}
J.~Baik and Z.~Liu,
\emph{Pinched-up periodic KPZ fixed point},
arXiv:2403.01624 (2024).

\bibitem{DDV24}
S.~Das, D.~Dauvergne, and B.~Vir\'ag,
\emph{Upper tail large deviations of the directed landscape},
arXiv:2405.14924 (2024).

\bibitem{DT24}
S.~Das and L.-C.~Tsai,
\emph{Solving marginals of the LDP for the directed landscape},
arXiv:2405.17041 (2024).

\bibitem{DOV22}
D.~Dauvergne, J.~Ortmann, and B.~Vir\'ag,
\emph{The directed landscape},
Acta Math. \textbf{229} (2022), no.~2, 201--285.
\href{https://doi.org/10.4310/ACTA.2022.v229.n2.a1}{doi:10.4310/ACTA.2022.v229.n2.a1}.


\bibitem{DZ25}
D.~Dauvergne and L.~Zhang,
\emph{Characterization of the directed landscape from the KPZ fixed point},
arXiv:2412.13032v3 (2025).

\bibitem{GH22}
S.~Ganguly and M.~Hegde,
\emph{Sharp upper tail behavior of line ensembles via the tangent method},
arXiv:2208.08922 (2022).

\bibitem{GHZ25}
S.~Ganguly, M.~Hegde, and L.~Zhang,
\emph{Brownian bridge limit of path measures in the upper tail of KPZ models},
arXiv:2311.12009v2 (2025).

\bibitem{GLLT23}
P.~Y.~Gaudreau Lamarre, Y.~Lin, and L.-C.~Tsai,
\emph{KPZ equation with a small noise, deep upper tail and limit shape},
Probab. Theory Related Fields \textbf{185} (2023), no.~3--4, 885--920.
\href{https://doi.org/10.1007/s00440-022-01185-2}{doi:10.1007/s00440-022-01185-2}.

\bibitem{JR21}
K.~Johansson and M.~Rahman,
\emph{Multitime distribution in discrete polynuclear growth},
Comm. Pure Appl. Math. \textbf{74} (2021), no.~12, 2561--2627.

\bibitem{LL25}
Y.~Liao and Z.~Liu,
\emph{Multipoint distributions of the KPZ fixed point with compactly supported initial conditions},
arXiv:2509.03246 (2025).

\bibitem{LT25}
Y.~Lin and L.-C.~Tsai,
\emph{Spacetime limit shapes of the KPZ equation in the upper tails},
Comm. Math. Phys. \textbf{406} (2025), Paper No.~113.
\href{https://doi.org/10.1007/s00220-025-05284-8}{doi:10.1007/s00220-025-05284-8}.

\bibitem{Liu22}
Z.~Liu,
\emph{Multipoint distribution of TASEP},
Ann. Probab. \textbf{50} (2022), no.~4, 1255--1321.
\href{https://doi.org/10.1214/21-AOP1557}{doi:10.1214/21-AOP1557}.

\bibitem{LW24}
Z.~Liu and Y.~Wang,
\emph{A conditional scaling limit of the KPZ fixed point with height tending to infinity at one location},
Electron. J. Probab. \textbf{29} (2024), 1--27.
\href{https://doi.org/10.1214/24-EJP1092}{doi:10.1214/24-EJP1092}.

\bibitem{LZ25}
Z.~Liu and R.~Zhang,
\emph{An upper tail field of the KPZ fixed point},
Comm. Math. Phys. \textbf{406} (2025), Paper No.~198.
\href{https://doi.org/10.1007/s00220-025-05375-6}{doi:10.1007/s00220-025-05375-6}.

\bibitem{MQR21}
K.~Matetski, J.~Quastel, and D.~Remenik,
\emph{The KPZ fixed point},
Acta Math. \textbf{227} (2021), no.~1, 115--203.
\href{https://doi.org/10.4310/ACTA.2021.v227.n1.a3}{doi:10.4310/ACTA.2021.v227.n1.a3}.


\bibitem{NZ22}
R.~Nissim and R.~Zhang,
\emph{Edgeworth-type expansion for the one-point distribution of the KPZ fixed point with a large height at a prior location},
arXiv:2210.04999 (2022).

\bibitem{Rah26}
M.~Rahman,
\emph{Revisiting the temporal law in KPZ random growth},
to appear in Ark. Mat.; arXiv:2504.19975v2 (2025).

\bibitem{RY99}
D.~Revuz and M.~Yor,
\emph{Continuous Martingales and Brownian Motion},
3rd ed., Grundlehren der mathematischen Wissenschaften, vol.~293,
Springer-Verlag, Berlin, 1999.

\end{thebibliography}
\end{document}